\documentclass{article}
\usepackage{mathptmx}
\usepackage{latexsym,amssymb}
\usepackage{a4}
\usepackage{amsthm}
\usepackage[svgnames]{pstricks}
\usepackage{pst-plot}
\usepackage{fancyhdr}
\usepackage[dvips]{graphicx}
\usepackage{subfigure}
\usepackage{hyperref}
\hypersetup{bookmarks=false,colorlinks=true, %
  linkcolor=DarkSlateGray,citecolor=DarkSlateGray}

\DeclareMathAlphabet{\mathcal}{OMS}{cmsy}{m}{n}

\def\CC{\mathbb{C}}

\def\RR{\mathbb{R}}

\def\NN{\mathcal{N}}
\def\ZZ{\mathbb{Z}}

\def\aa{\mathbf{a}}
\def\bb{\mathbf{b}}

\def\ad{\mathrm{ad}}
\def\coad{\mathrm{coad}}
\def\dd{\mathrm{d}}
\def\ee{\mathrm{e}}
\def\ii{\mathrm{i}}

\def\M{\mathcal{M}}

\def\gg{\mathfrak{g}}
\def\ssl{\mathfrak{sl}}
\def\so{\mathfrak{so}}
\def\se{\mathfrak{se}}

\def\D{\mathbf{D}}
\def\J{\mathbf{J}}
\def\SE{\mathsf{SE}}
\def\SL{\mathsf{SL}}
\def\SO{\mathsf{SO}}
\def\OO{\mathsf{O}}

\def\half{{\textstyle\frac12}}
\def\quarter{{\textstyle\frac14}}
\def\tfrac#1#2{{\textstyle\frac{#1}{#2}}}

\makeatletter
  \@addtoreset{figure}{section}
  \@addtoreset{table}{section}
  \@addtoreset{equation}{section}

\makeatother
\newtheorem{theorem}{Theorem} %[section]
\newtheorem{lemma}[theorem]{Lemma}
\newtheorem{proposition}[theorem]{Proposition}
\newtheorem{corollary}[theorem]{Corollary}
\newtheorem{conjecture}[theorem]{Conjecture}
\theoremstyle{definition}
\newtheorem{remark}[theorem]{Remark}

\begin{document}

\title{Deformation of geometry \\ and bifurcations of vortex rings}

\author{James Montaldi\footnote{School of Mathematics, University of Manchester, Oxford Rd, Manchester  M13 9PL, England. \texttt{j.montaldi@manchester.ac.uk}} \ \& 
Tadashi Tokieda\footnote{
Trinity Hall, Cambridge CB2 1TJ, England.  \texttt{t.tokieda@dpmms.cam.ac.uk}}
}

\maketitle

\begin{abstract}
We construct a smooth family of Hamiltonian systems, together with a family of group symmetries and momentum maps, for the dynamics of point vortices on surfaces parametrized by the curvature of the surface.  Equivariant bifurcations in this family are characterized, whence the stability of the Thomson heptagon is deduced without recourse to the Birkhoff normal form, which has hitherto been a necessary tool. 
\end{abstract}

\section*{Introduction}

   The present paper introduces one geometric idea and implements it in one dynamical problem.  

Here is the problem.  On the Euclidean plane, a ring of identical point vortices shaped as a regular $n$-gon is a relative equilibrium, in that it spins while keeping the same shape.  Is this solution stable if we perturb the initial shape away from the regular $n$-gon?  When $n < 7$, 
linear stability analysis (first carried out by Thomson \cite{Thomson}) concludes that the solution is stable; when $n > 7$, it is likewise unstable.  But when $n = 7$, degeneracy makes linear analysis inapplicable and prevents us from concluding.  Is the regular $7$-gon stable or unstable?  This is known as the {\it Thomson heptagon\/} problem.    It has been answered in the affirmative, see  \cite{Mertz,KY02,Schmidt04}, although as pointed out in \cite{KY02} the argument in \cite{Mertz} is incomplete. Indeed part of our approach could be viewed as completing the argument of \cite{Mertz}.
   
   The spirit of the approach goes back to Poincar\'e.  In the theory of dynamical systems, the simplest solution methods to problems require some nondegeneracy condition (nonzero determinant, nonresonant frequencies, $\ldots$).  When the problem, call it $\cal P$, is degenerate, we have to mobilize heavy machinery (cf.\ \cite{Schmidt04} for a proof with recourse to the Birkhoff normal form that the Thomson heptagon is nonlinearly stable, as well as for historical details).  But there is an alternative approach.  Embed ${\cal P} = {\cal P}_0$ in a parametric family ${\cal P}_{\lambda}$ of problems and deform it away from degeneracy.  The problem $\cal P$ can become tractable when regarded as $\lim_{\lambda \to 0} {\cal P}_{\lambda}$, if we happen to understand well enough the bifurcations that occur in such a deformation.  Thus, this approach trades one machinery for another, of bifurcation theory.  The point is that the latter sometimes sheds an 
unusual light on the problem compared to the former.
   
In the classic applications of this idea, people deform the dynamical system by adding a perturbation term in $\lambda$.  This, however, we cannot do in our special instance of the heptagon problem if we wish to preserve the hydrodynamic motivation: as long as we are studying the dynamics on the plane, it makes little physical sense to tamper with the Hamiltonian.     

   We therefore deform not so much the dynamical system {\it but rather the phase space\/} on which the system evolves.  A deformed choice of the phase space fixes canonically, by the hydrodynamic motivation, a deformed Hamiltonian formalism.  Explicitly, we take $\lambda$ to be the Gaussian curvature\footnote{Actually formulaic convenience leads us to take $4\lambda$ to be the Gaussian curvature.} and deform the original plane to $\lambda > 0$ (family of spheres) and to $\lambda < 0$
 (family of hyperbolic planes).  Corresponding to this family of surfaces parametrized by $\lambda$, we must write a whole parametric family of Hamiltonian systems for point vortices: a family of symplectic (K\"ahler) forms depending on $\lambda$, a family of symmetry groups and momentum maps depending on $\lambda$, a family of invariant Hamiltonians depending on $\lambda$---all dependences arranged to be smooth.  We do this in section 1.  In section 2 we carry out the stability analysis for the parametric family.  Bifurcations are characterized, and the nonlinear stability of the heptagon is deduced, in section 3.
 
    The idea of {\it deforming the geometry\/} underlying the dynamics of point vortices, in particular as a route to a better understanding of the Thomson heptagon problem, arose during an evening conversation between the two authors in Peyresq, in the summer of 2003.  We have since discussed it in seminars and conferences, and part of it has leaked into the literature \cite{Boatto}.  We set down the full story in the present paper. The stability of the Thomson heptagon is stated below as Corollary \ref{coroll:Thomson}.

\section{Smooth family of geometries}

In this section we describe the family of surfaces of constant curvature, containing the hyperbolic planes ($\lambda<0$), the Euclidean plane ($\lambda=0$), and the spheres ($\lambda>0$).  Each is a homogenous space, i.e.\ an orbit of its group of symmetries, and for the different signs of $\lambda$ we describe the different groups.  We begin with the 1-parameter family of Lie algebras $\gg_\lambda$ (rather than groups) in section \ref{sec:Lie algebra}, and in order to obtain the corresponding family of surfaces $\M_\lambda$, we shift the usual linear action of these Lie algebras to obtain affine linear actions.  In \ref{sec:surfaces} the geometry of $\M_\lambda$ is described; among other things it has curvature $4\lambda$.  The Hamiltonian for point vortices on  $\M_\lambda$ is based on Green's function for the Laplacian on $\M_\lambda$, and we meet three possible choices of Green's function according to choices of the `boundary condition'.  In \ref{sec:vorticity} we comment on the implications of the different choices of Green's function.

%%%%%%%%%%%%%%%
\subsection{Lie algebras} \label{sec:Lie algebra}

Let $\lambda \in \RR$.  On $\RR^3$ with coordinates $(x,y,u)$, consider the family of metrics
\begin{equation} \label{eq:family of metrics}
\dd s^2 = \dd x^2 + \dd y^2 + \lambda\,\dd u^2.
\end{equation}
The Lie group of linear transformations preserving
this metric will be denoted $\SO(q_\lambda)$, where $q_\lambda =
\mathrm{diag}[1,1,\lambda]$ is the metric tensor.   For the Lie algebra, we have $X\in
\so(q_\lambda)$ if and only if $X^Tq_\lambda + q_\lambda X=0$. A
basis for $\so(q_\lambda)$ is
\begin{equation}\label{eq:coadjoint rep}
X_1 =
\pmatrix{0&0&0\cr 0&0&-\lambda\cr 0&1&0},\quad
  X_2 = \pmatrix{0&0&\lambda\cr 0&0&0\cr -1&0&0},\quad
  X_3 = \pmatrix{0&-1&0\cr 1&0&0\cr 0&0&0}.
\end{equation}
(Strictly speaking, the third column of $X_i$ is
arbitrary when $\lambda=0$, meaning the family of all automorphisms of
(\ref{eq:family of metrics}) is not \emph{flat}.  We are picking a
component which \emph{is} a flat family over $\lambda$.) This basis satisfies the commutation relations
$$
[X_1,\, X_2] = \lambda X_3,\quad [X_2,\, X_3]=X_1,\quad [X_3,\, X_1]
= X_2.
$$
From now on, we shall abbreviate $G_\lambda = \SO(q_\lambda)$ and
$\gg_\lambda=\so(q_\lambda)$.

The Lie algebra $\gg_\lambda$ is isomorphic to
$\so(3)$ for $\lambda>0$, to $\se(2)$ for $\lambda=0$, and to $\ssl(2)$ for $\lambda<0$.
Indeed, for $\lambda\neq0$, the standard commutation relations are
recovered by rescaling the basis to $\{|\lambda|^{-1/2}X_1,\,
|\lambda|^{-1/2}X_2,\, X_3\}$.  So
$$
G_\lambda\simeq \cases{\SO(3) & if $\lambda>0$\cr
\SE(2) & if $\lambda=0$\cr
\SL(2,\RR) & if $\lambda<0$.}
$$

It is seen from the commutation relations that in
the adjoint representation of $\gg_\lambda$ the basis elements are
represented by $\ad_{X_j} = -X_j^T$; in other words, if
$\sum_ja_jX_j\in\gg_\lambda$ is written as a vector
$\mathbf{u}=(a_1\;a_2\;a_3)^T$,  then $\ad_{X_j}(\mathbf{u}) = -X_j^T\mathbf{u}$.  
In the \emph{coadjoint\/} representation the basis elements
$X_j\in\gg_\lambda $ are represented by the matrices $X_j$
themselves. Thus the original $\RR^3$ from which we started may be
naturally identified with $\gg_\lambda^*$.

\paragraph{Affine action}
Whereas the coadjoint action defined by the matrices $X_j$ depends
continuously on $\lambda$, it is not possible to track a single
orbit  continuously as $\lambda$ crosses 0.  Yet, in what follows, we wish to do
just that.  To this end we must shift the linear coadjoint action by a translation, making it an affine action.

The affine action of the Lie algebra is given by
\begin{equation}\label{eq:infinitesimal action}
X\cdot\mu  = X\mu+\tau(X),
\end{equation}
where $X\mu$ is the linear action part (matrix times vector) and
$$
\tau(aX_1+bX_2+cX_3) = \pmatrix{-b/2\cr a/2\cr0}
$$
is the translation.  The orbit we track is the one through the origin,  
cf.\ section \ref{sec:surfaces}.

\begin{remark}
Our translation $\tau$, a function of the element of
$\gg_{\lambda}$, is a 1-cocycle taking values in $\gg_{\lambda}^*$,
a {\it symplectic cocycle\/} in Souriau's terminology \cite{Souriau} because the matrix of $\tau$ is
skew-symmetric.  It is known that every cocycle is exact when the group is semi-simple, and our $G_{\lambda}$ is semi-simple for $\lambda\neq0$.
Here
$\tau=\delta((0,\,0,\, 1/2\lambda)^T)$, since by definition
$\delta(\mu)(X) = -\coad_X\mu = -X\mu$.
The natural invariant Poisson structure on $\RR^3=\gg_\lambda^*$
with the cocycle $\tau$ is given by (cf.\ \cite{MaRa,Souriau})
\begin{equation}\label{eq:PoissonBracket}
\{f,\,g\}(\mu) = \left<\mu,\; [ \dd f(\mu), \dd g(\mu) ]\right> -
\left<\tau(\dd f(\mu)), \dd g(\mu)\right>
\end{equation}
under the identification $\dd f(\mu), \dd g(\mu)\in(\gg_\lambda^*)^*\simeq\gg_\lambda$.  The Casimir for this Poisson structure is $x^2 + y^2 + \lambda u^2 - u$, so level sets of this function are the orbits of the shifted coadjoint action (\ref{eq:infinitesimal action}), of which we shall take advantage below.  For the record, the Kostant-Kirillov-Souriau symplectic form on the affine coadjoint orbits is given by the same formula,
$$
\Omega_\mu(\mathbf{u},\mathbf{v}) = \left<\mu,\,[\xi,\eta]\right>
  - \left<\tau(\xi),\eta\right>,
$$
where $\mathbf{u}=\coad_\xi\mu$ and $\mathbf{v}=\coad_\eta\mu$.
\end{remark}

%%%%%%%%%%%%%%%%%%%
\subsection{Surfaces}\label{sec:surfaces}

Now consider the family of quadratic surfaces through the
origin in $\RR^3$,
\begin{equation} \label{eq:family of surfaces}
x^2 + y^2 + \lambda u^2 - u = 0.
\end{equation}
When $\lambda>0$, this looks like an ellipsoid with centre at $(x,y,u) =
(0,0,1/2\lambda)$. With the metric (\ref{eq:family of metrics}), however, this ellipsoid-looking surface is in fact a sphere of radius $1/2\sqrt\lambda$.  Its Gaussian curvature is $4\lambda$.

When $\lambda=0$,  (\ref{eq:family of surfaces}) defines
the paraboloid $u=x^2+y^2$, and the metric is the usual metric on
the $xy$-plane lifted to the paraboloid by orthogonal
projection, so is of curvature 0.

When $\lambda<0$, the metric (\ref{eq:family of metrics}) becomes
Lorentzian, but restricted to either sheet of the 2-sheeted
hyperboloid defined by (\ref{eq:family of surfaces}) it induces the
hyperbolic metric of constant negative curvature $4\lambda$; we
consider just the `upper sheet' that passes through the origin, see Figure \ref{fig:family of surfaces}.

We refer to the surface (\ref{eq:family of surfaces}) with metric induced from (\ref{eq:family of metrics}) 
as $\M_\lambda$.  It is easy to check that $\M_\lambda$ is invariant under the infinitesimal action (\ref{eq:infinitesimal action}), and is therefore an orbit of the affine coadjoint $G_\lambda$-action on $\gg_\lambda^*$.

To create a {\it uniform coordinate system\/} on $\M_\lambda$, we
use stereographic projection on the $xy$-plane, centered at
the point $(0,0, 1/\lambda)$ (where  $\M_\lambda$ intersects the $u$-axis, besides the origin); for $\lambda=0$
this is the orthogonal projection. We also identify the $xy$-plane with $\CC$, via $z = x+\ii y$. The map inverse to the projection has the formula
\begin{equation} \label{eq:stereographic}
 z \mapsto
\pmatrix{x+\ii y\cr u} = \frac{1}{1+\lambda|z|^2}\pmatrix{z\cr
|z|^2}.
\end{equation}
The domain of this map is %$\D_\lambda =
$\{z\in\CC \mid 1+\lambda|z|^2>0\}$, which is the entire plane if $\lambda\geqslant 0$ and
a bounded disc (Poincar\'e disc) if $\lambda<0$. For the sphere,
the equator corresponds to $|z|^2 = 1/\lambda$, while the point antipodal to $z$ is 
$- 1/\lambda\bar z$.

\begin{figure}[t]
 \psset{unit=1.4}
\begin{center}
\fbox{\begin{pspicture}(-2,-1.2)(2,2.8)
 \psline[linecolor=gray](-2,-0.5)(1,-0.5)(1.8,0.5)(-1,0.5)(-2,-0.5)
 \psellipse[fillstyle=solid,fillcolor=white](0,1.2)(1,1.2)
 \parametricplot[linestyle=dashed]{0}{180}{t cos 1.2 t sin 0.2 mul add}
 \parametricplot{0}{180}{t cos 1.2 t sin -0.2 mul add}
 \psline[linewidth=0.2pt]{->}(0,-0.2)(0,2.8)
%  \psline[linewidth=0.2pt,arrowsize=.1]{->}(0,2.4)(-0.5,-0.3)\rput(-0.33,0.6){$\times$}
 \psline[linewidth=0.5pt,linestyle=dotted](0,2.4)(-0.33,0.6)\rput(-0.33,0.6){$\times$}
 \psline[linewidth=0.2pt,arrowsize=.1]{->}(-0.33,0.6)(-0.5,-0.3)
 \psdot(0,0) % origin
 \psdot[linecolor=gray](0,2.4) % other pole
 \rput(0,-1){$\lambda>0$ (sphere)}
\end{pspicture}}
% \fbox{\begin{pspicture}(-2,-1.85)(2,1.85)
%  \psline(-2,-0.5)(1,-0.5)(2,0.5)(-1,0.5)(-2,-0.5)
%  \psdot(0,0)
%  \rput(0,-1.65){$\lambda=0$}
% \end{pspicture}}\hfill
\fbox{\begin{pspicture}(-2,-1.7)(2,2.3)
 \psline[linecolor=gray](-2,-0.55)(1,-0.55)(1.8,0.45)(-1,0.45)(-2,-0.55)
 \psline[linewidth=0.2pt](0,-1.2)(0,0) % lower part of axis
 \psplot[fillstyle=solid,fillcolor=white]{-1.7}{1.7}{x 2 exp 2 mul 0.25 add sqrt 2 div 0.25 sub}
 \psellipse[linewidth=0.5pt](0,1)(1.7,0.3)
 \psline[linewidth=0.2pt]{->}(0,0.75)(0,2) % upper part of axis
%% for stereographic projection: lower hyperboloid + line of projection
 \psplot[linestyle=dashed,linewidth=0.2pt]{-1}{1}{0.25 x 2 exp 2 mul 0.25 add sqrt 2 div add -1 mul}
 \psellipse[linewidth=0.2pt,linestyle=dashed](0,-0.8)(0.7,0.1)
 \psline[linewidth=0.2pt,arrowsize=.1]{->}(-0.9,0.6)(-0.27,-0.17)\rput(-0.9,0.6){$\times$}
 \psline[linewidth=0.5pt,linestyle=dotted](-0.27,-0.17)(0,-0.5)
 \psdot[linecolor=black](0,0) % origin
 \psdot[linecolor=gray](0,-0.5) % other pole
 \rput(0,-1.5){$\lambda<0$ (hyperboloid)}
\end{pspicture}}
 \caption{Geometries in the 1-parameter family}
 \label{fig:family of surfaces}
\end{center}
\end{figure}
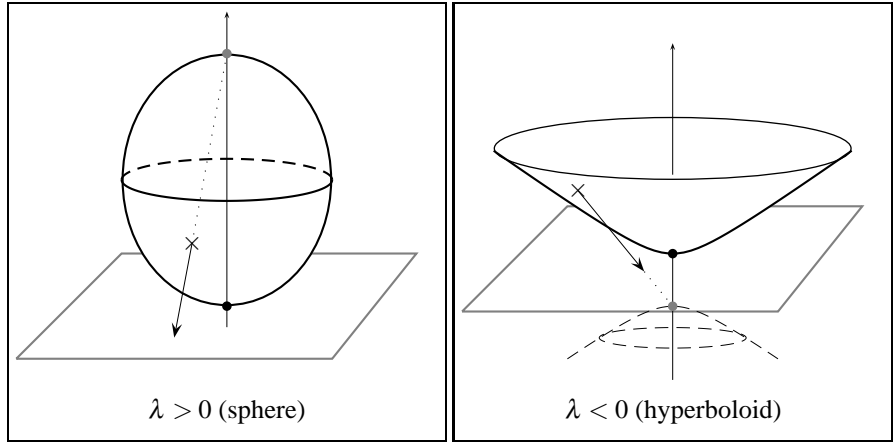

The metric on the surface $\M_\lambda$ induced from that in (\ref{eq:family of
metrics}), in terms of the complex variable $z$, is 
$$\dd s^2 = \frac{1}{\sigma^2}|\dd z|^2,$$
where $\sigma = 1+\lambda|z|^2$, a notation we shall use throughout.  
The circle $|z|=r$ in $\CC$ maps to a circle of radius $a$ on $\M_\lambda$, where\footnote{Despite three formulae, $a$ is a single analytic function of  $r, \lambda$, with series expansion $a=r-\frac13r^3\lambda+\frac15r^5\lambda^2-\frac17r^7\lambda^3+\cdots$ convergent for $|r^2\lambda|<1$.}
\begin{equation}\label{eq:radius}
a=\cases{\frac{1}{\sqrt{\lambda}}\tan^{-1}\left(r\sqrt\lambda\right) & if $\lambda>0$\cr
 r & if $\lambda=0$\cr
 \frac{1}{\sqrt{- \lambda}}\tanh^{-1}\left(r\sqrt{-\lambda}\right)& if $\lambda<0$.}
\end{equation}

Pulling back the vector fields $X_j$ of the Lie algebra (or rather
their affine variants shifted by $\tau$) via the
stereographic projection yields
$$\xi_1(x,y) = \half\left(2\lambda xy,\; 1-\lambda(x^2-y^2)\right),\quad
  \xi_2(x,y) = -\half\left(1+\lambda(x^2-y^2),\; 2\lambda xy\right),$$
$$  \xi_3(x,y) = (-y, x),
$$
or in complex variables
$$
\xi_1+\ii\xi_2 = - \ii \partial_{\bar z}
-\ii\lambda\,z^2\partial_z, \quad \xi_3 = \ii\left(z\,\partial_z -
\bar z\,\partial_{\bar z}\right),
$$
or in polar coordinates
$$
\xi_1+\ii\xi_2 = 
 \half \ee^{\ii\theta}\, \frac{1-\lambda r^2}{r}\, \partial_\theta + \half\, \ii\, \ee^{\ii\theta}\sigma\,\partial_r,
 \quad 
 \xi_3 = \partial_\theta.
$$

\paragraph{Symplectic structures}
Up to a scalar multiple, there exists a unique $\SO(q_\lambda)$-invariant symplectic form on $\M_\lambda$. We choose the scalar so that 
\begin{equation}\label{eq:family of symplectic forms}
 \Omega_{\lambda} = \frac2{\sigma^2}\dd x\wedge \dd y = \frac{\ii}{\sigma^2}\dd z\wedge \dd\bar z.
\end{equation}
The choice of scaling is such that the sphere $\M_\lambda$ of radius
$1/2\sqrt{\lambda}$ acquires symplectic area $2\pi\lambda^{-1}$.  With respect to the
basis $\{X_1,X_2,X_3\}$ for the Lie algebra, the momentum map
takes the form
\begin{equation}\label{eq:family of momentum maps}
  \J_\lambda(z) = \frac{1}{\sigma}(z, |z|^2).
\end{equation}
This coincides with the inclusion $\M_\lambda \hookrightarrow \RR^3$ given in (\ref{eq:stereographic}), which shows that $\Omega_{\lambda}$ coincides with the KKS symplectic form on the affine coadjoint orbit.

\paragraph{Green's functions}
The metric  (\ref{eq:family of metrics}) on $\RR^3$ induces
the metric on $\M_\lambda$. In terms of the uniform coordinate system
(\ref{eq:stereographic}), the metric tensor
is $\sigma^{-2}\mathrm{diag}[1, 1]$.  The Laplace-Beltrami operator on
$\M_\lambda$ is 
$$
\Delta f = \sigma^2\!\left(\frac{\partial^2}{\partial x^2} + \frac{\partial^2}{\partial
y^2}\right)f = \quarter\sigma^2\frac{\partial^2}{\partial z\partial
\bar z}\, f.
$$
The 2-point Green's function for this operator is
\begin{equation} \label{eq:green's function pole at infinity}
G(z;w) = \log |z-w|^2.
\end{equation}
This satisfies $\Delta_z G =0$ for $z\neq w$ and has a logarithmic singularity at $z=w$.  When we regard the plane as a model for (most of) the sphere, $G$ has another singularity at $z=\infty$.

An alternative Green's function is
\begin{equation} \label{eq:green's function opposite vortices}
G_1(z;w) = \log \frac{|z-w|^2}{|1+\lambda z\bar w|^2}.
\end{equation}
This satisfies $\Delta_zG_1(z;w)=0$ for $z\not\in \{w,\, -1/\lambda\bar w\}$ and has a logarithmic singularity at those excluded points (which are antipodal to each other), but is regular at $z=\infty$.

The Green's function that is usually used on the sphere is the `log of Euclidean distance', whose expression after stereographic projection is
\begin{equation} \label{eq:green's function from sphere}
G(z;w) =
\log \frac{|z-w|^2}{(1+\lambda|z|^2)(1+\lambda|w|^2)}.
\end{equation}
Away from the pole at $z=w$, this satisfies $\Delta_zG(z;w) = -4\lambda$, so it is not, if we go by
the book, a Green's function. This function is regular at $z=\infty$, and also has a well-defined limit as $w\to\infty$, namely $G(z;\infty) = - \log\sigma$ (up to an additive constant of $-\log\lambda$), which checks against $\Delta_z(-\log\sigma)=-4\lambda$.

In the next section we comment on the differences among these Green's functions in the context of the dynamics of point vortices.

%%%%%%%%%%%%%%%%%%%
\subsection{Hamiltonians for point vortices} \label{sec:vorticity}

   We recall how the Hamiltonian formalism for the the dynamics of point vortices works.  
   
   Let $(u, v)^T$ be the velocity field of an inviscid, incompressible flow on a domain $D
\subseteq \RR^2 \simeq \CC$.  The incompressibility 
$\frac{\partial}{\partial x} u + \frac{\partial}{\partial y} v = 0$ implies the existence of a {\it stream function\/}
$\psi : D \to \RR$ such that 
\begin{equation}\label{eq:Hamilton's equations}
u = \frac{\partial}{\partial y} \psi, \quad
v = -\frac{\partial}{\partial x} \psi.
\end{equation}
The curl of the velocity\footnote{The minus sign makes $-\Delta$ a positive operator.  But we shall be casual about the sign and use $+\Delta$
as well as $-\Delta$.  All that the casualness causes is to reverse the direction of the flow.} is 
$\frac{\partial}{\partial x} v - \frac{\partial}{\partial y} u = - \Delta \psi$.
The boundary condition for an inviscid flow is that $(u, v)^T$ be tangent to $\partial D$
everywhere, equivalently that every connected component of $\partial D$ be a level set
of $\psi$.  The total circulation along all the boundary components is, by Stokes's theorem,
\begin{equation}\label{eq:total circulation}
\int_{\partial D} u\,\dd x + v\,\dd y = \int_D -\Delta \psi\, \dd x\dd y, 
\end{equation}
the total curl present on $D$.  

   Now consider a model situation where the flow is generated by a curl concentrated at a 
singularity $z_0 = x_0 + \ii y_0$ and the circulation around that singularity is $2\pi \kappa\,$:
$$
-\Delta \psi(z) = 2\pi \kappa \delta(z - z_0).
$$
We recognize that, up to the sign and a scalar coefficient, $\psi$ is Green's function.  We say that we have a {\it point vortex\/} of {\it vorticity\/} $\kappa$ at $z_0$.  In the situation where we have
point vortices of vorticities $\kappa_1, \ldots, \kappa_n$ at $z_1, \ldots, z_n$, each vortex moves carried by the sum of the flows generated by all the other vortices.  The equations 
(\ref{eq:Hamilton's equations}) show that the dynamics then is Hamiltonian. 

   The theory is written analogously on any domain of any Riemann surface.  In particular, on
 the surfaces $\M_\lambda$ discussed above, the Hamiltonian is
$$
H_\lambda(z_1,\dots,z_N) =
-\frac{1}{4\pi}\sum_{i<j}\kappa_i\kappa_jG_\lambda(z_i;z_j).
$$
Note an unusual feature of this Hamiltonian system: unlike in classical mechanics, the phase space here is not the cotangent bundle of anything, but rather the $n$-fold product $\M_\lambda \times \cdots \times \M_\lambda$ (minus the diagonals if we want {\it a priori\/} to avoid collisions) with the 
weighted-sum symplectic form $\kappa_1 \Omega_\lambda \oplus \cdots \oplus \kappa_n \Omega_\lambda$.

   For $\lambda > 0$, $\M_\lambda$ is compact without boundary.  In this case, the fact of nature 
(\ref{eq:total circulation}) forces the total vorticity to be zero: $\sum_j \kappa_j = 0$.  This, in principle, bans placing a lone point vortex on $\M_\lambda$ for $\lambda > 0$ or on any closed Riemann surface.  The dodge around this ban, favoured in the literature, is to impose a constant {\it background vorticity\/}
$$
- \mathrm{sum~of~circulations/area} = - \lambda \sum_j \kappa_j ,
$$
which results in (\ref{eq:green's function from sphere}).  For a (geo)physical example, a rigidly rotating sphere entails such a background vorticity.  But even if background vorticity dodges around $\lambda > 0$, continuing it to $\lambda \leqslant 0$ gets us into trouble, for on these noncompact surfaces, background vorticity imparts an infinite amount of energy to the flow.  So for $\lambda \leqslant 0$ the family (\ref{eq:green's function pole at infinity}) seems preferable.

However, welding together (\ref{eq:green's function from sphere}) for $\lambda\geqslant 0$, (\ref{eq:green's function pole at infinity}) for $\lambda<0$ has a decisive defect: the resulting family is {\it not smooth\/} in $\lambda$. There are three options for having a smooth family.

\begin{description}
\item[1.] Use (\ref{eq:green's function pole at infinity}).  This costs postulating an immobile vortex of vorticity $-\sum_j \kappa_j$ at the North Pole for $\lambda>0$.

\item[2.] Use (\ref{eq:green's function opposite vortices}).  This costs introducing `counter-vortices'  at antipodes to the $z_j$s  for $\lambda > 0$ (recall that the point antipodal to $z$ is $-1/\lambda\bar z\,$).

\item[3.] Use (\ref{eq:green's function from sphere}). This costs infinite energy for $\lambda \leqslant 0$.
\end{description}
In the planar case $\lambda=0$ all three Green's functions 
(\ref{eq:green's function pole at infinity}), (\ref{eq:green's function opposite vortices}), 
(\ref{eq:green's function from sphere})
agree.

In this paper, we opt for the family defined in (\ref{eq:green's function from sphere}) with the constant background vorticity, because the infinite energy of a tame flow would not shock any fluid dynamicist---flows on the plane uniform at infinity and such are handled routinely---whereas the smoothness of the family, without postulating extraneous objects, is essential for us.  In section \ref{sec:other Hamiltonians} we sketch how the analysis can be adapted to the other two options.

%%%%%%%%%%%%%%%%%%%%%%%%%%%%%%%%%%%%
\section{Nondegenerate analysis of vortex rings}
\label{sec:vortices}

In this section we study a ring of $n$ vortices with identical vorticities $\kappa$, shaped as an $n$-gon on $\M_{\lambda}$.  We evaluate the Hamiltonian, the momentum map, and the augmented Hamiltonian at the regular ring in section \ref{sec:regular ring}.  We obtain the Hessian of the augmented Hamiltonian and its spectral data in \ref{sec:Hessians}.  In \ref{sec:symplectic slice} we construct the symplectic slice and say what we can, as far as this linear analysis goes, about the stability of the ring.  The answers will depend not only on $n$ but
also on $\lambda$.

Recall that the augmented Hamiltonian is given by $H-\omega J$, where $J=J_\lambda$ is the conserved quantity coming from the rotational symmetry; it represents the Hamiltonian in a frame rotating with angular
velocity $\omega$.  Its critical points are therefore equilibria in the rotating frame, i.e.\ relative equilibria, and the Hessian restricted to the symplectic slice determines the stability.

\subsection{Regular ring}\label{sec:regular ring}

Let $\gamma$ be a primitive $n\,$th root of unity.  When the vortices are placed at the vertices  $z_j=r\ee^{2\pi\ii j/n}$ ($j = 1, \ldots, n$) of a regular $n$-gon of radius $r$, the Hamiltonian $H_\lambda$ takes the value 
$$
  h_\lambda(r^2) =
  -\frac{n\kappa^2}{8\pi}\;\sum_{j=1}^{n-1}G_\lambda(r,r\gamma^{\, j}),
$$
where $G_\lambda=\log\, \circ\, \rho_\lambda$ with 
$\rho_\lambda(r,r\gamma^{\, j}) = |1-\gamma^{\, j}|^2 r^2/(1+\lambda r^2)^2$. 
In view of the identity $\prod_{j=1}^{n-1}(1-\gamma^{\, j}) = n$, 
$$
 \prod_{j=1}^{n-1}\rho_\lambda(r, r\gamma^{\, j}) = n^2 \left( \frac{r}{1+\lambda
 r^2} \right)^{\! 2(n-1)},
$$
hence
$$
    h_\lambda(r^2) = -\frac{n(n-1)\kappa^2}{8\pi}\log \frac{r^2}{(1+\lambda
    r^2)^2} +\mathrm{const.}
$$
recall the momentum map given in (\ref{eq:family of momentum maps}). The first component vanishes for these regular rings, while the value of the second component at the regular ring is
$$
J_\lambda(r^2) = \frac{n\kappa r^2}{1+\lambda r^2}.
$$   
(Notice that the full momemtnum map is typeset in bold, while this component is not.) Ignoring the constant, the augmented Hamiltonian then takes the value
\begin{eqnarray*}
\widehat{h}_{\lambda} (r^2) &=& h_\lambda(r^2) - \omega J_\lambda(r^2)\\
 &=& -\frac{n(n-1)\kappa^2}{8\pi}\log \frac{r^2}{(1+\lambda
    r^2)^2} - \omega\frac{n\kappa r^2}{1+\lambda r^2},
\end{eqnarray*}
which admits a critical point at $r=r_0\neq0$ if and only if
\begin{equation}\label{eq:angular velocity}
\omega = \omega_{\, 0} = - \frac{(n-1)\kappa}{8\pi}\; \frac{1-\lambda^2 r_0^4}{r_0^2}.
\end{equation}
This is the angular velocity of the regular ring.
(It is a little surprising that $\omega_{\, 0}$ is even in $\lambda$.)

%%%%%%%%%%%%%%%%

\subsection{Hessians}\label{sec:Hessians}
We continue with the system of $n$ identical point vortices, all with vorticity $\kappa$. In the uniform coordinate system on $\M_{\lambda}$, the Hamiltonian with Green's function (\ref{eq:green's function from sphere}) is
$$
H_\lambda(z_1,\dots,z_n) = -\frac{\kappa^2}{4\pi}\sum_{i<j}
\log \frac{|z_i-z_j|^2}{\sigma_i\sigma_j},
$$
where $\sigma_j = 1+\lambda|z_j|^2$ ($j=1,\ldots, n$). The augmented Hamiltonian then is 
$$
\widehat{H}_{\lambda}(z_1,\ldots, z_n) = 
H_\lambda(z_1,\ldots, z_n) -
\omega\sum_{j = 1}^n \kappa\frac{|z_j|^2}{\sigma_j}.
$$
We saw above that this is critical at $z_j = r_0\ee^{2\ii\pi j/n}$ and 
$\omega = \omega_{\, 0}$ as in (\ref{eq:angular velocity}). 
The entries in the Hessian are
\begin{equation}\label{eq:hessian}
\begin{array}{rclccrcl} %
\displaystyle\frac{\partial^2}{\partial r_j^2}\widehat{H}_{\lambda}
  &=& A &&&
\displaystyle\frac{\partial^2}{\partial r_j\partial r_k}\widehat{H}_{\lambda} &=&
 \displaystyle\frac{\kappa^2}{4\pi r_0^2 \left(1-\cos \frac{2\pi(j - k)}{n} \right)} \\[10pt]
\displaystyle\frac{\partial^2}{\partial\theta_j^2}\widehat{H}_{\lambda} &=&
  \displaystyle\frac{\kappa^2}{24\pi}(n^2-1)&&&
\displaystyle\frac{\partial^2}{\partial\theta_j\partial\theta_k}\widehat{H}_{\lambda} &=&
  -\displaystyle\frac{\kappa^2}{4\pi \left(1-\cos \frac{2\pi(j - k)}{n} \right)}\\[10pt]
&&\multicolumn{2}{l}{\displaystyle\frac{\partial^2}{\partial r_j\partial\theta_k}\widehat{H}_{\lambda}} &=\quad 0 & \multicolumn{2}{l}{\mbox{($\forall j, k$).}}
\end{array}
\end{equation}
At the critical point, and with $\omega=\omega_{\, 0}$, we have
\begin{equation}\label{eq:A}
A = \frac{(n-1)\kappa^2}{24\pi r_0^2\sigma^2}\left( (5-n)\sigma^2 + 6\widetilde{\sigma}^{\, 2}\right),
\end{equation} % double checked again!
where
$$
\sigma=1+\lambda r_0^2, \quad \widetilde{\sigma}=1-\lambda r_0^2
$$ 
(so that $\sigma+\widetilde{\sigma}=2$).   We used the identity, valid for 
$0\leqslant \ell\leqslant n\,$,  
cf.\ \cite{Hansen}:
\begin{equation}\label{eq:identity}
\sum_{j=1}^{n-1} \frac{\cos(2\pi\ell j/n)}{1-\cos(2\pi
  j/n)}\ =\ \frac{1}{6}(n^2-1)-\ell(n-\ell).
\end{equation}

%%%%%%%%%%%%%%
\paragraph{Eigenvalues of the Hessian}
The Hessian matrix $\dd^2\widehat{H}_{\lambda}$ is block-diagonal, with two $n\times n$ blocks both of which are symmetric circulant, so the eigenvalues can be written down at once.

Following the notation in \cite{LMR}, for $\ell=0,1,\dots,\lfloor n/2 \rfloor$ define the \emph{Fourier tangent vectors}
\begin{equation}\label{eq:Fourier basis}
\begin{array}{rcrcl}
  \zeta^{(\ell)}_r &=& \alpha^{(\ell)}_r + \ii\beta^{(\ell)}_r
  &=&
  \displaystyle\sum_{j=1}^n\exp(-2\pi\ii\,\ell j/n )\,\delta r_j
  \\[12pt]
  \zeta^{(\ell)}_\theta &=& \alpha^{(\ell)}_\theta + \ii\beta^{(\ell)}_\theta
  &=&
  \frac{1}{r_0}\,\displaystyle\sum_{j=1}^n\exp(-2\pi\ii\,\ell j/n)\, \delta\theta_j .
\end{array}
\end{equation}
Here $\delta r_j$ denotes the unit tangent vector in the $r_j$-direction, and similarly for 
$\delta\theta_j/r_0$.  For $\ell=0,\, n/2$, we have $\beta^{(\ell)}_r = \beta^{(\ell)}_\theta=0$ and the $\zeta^{(\ell)}$s are real. The $\alpha$s and $\beta$s form a set of $2n$ linearly independent vectors, forming a basis for the tangent space. For the record,
$$\delta r_j = \frac1n\sum_{\ell=1}^n \exp(2\pi\ii\,\ell j/n)\zeta_r^{(\ell)},$$
and similarly for $\delta\theta_j$.
The span $V_\ell = \left< \alpha^{(\ell)}_r, \alpha^{(\ell)}_\theta,\beta^{(\ell)}_r, \beta^{(\ell)}_\theta\right>$
is a subspace of {\it Fourier modes\/}; for $\ell=0,\, n/2$ they are 2-dimensional, while for all other indices $\ell$ they are 4-dimensional.

As each block of $\dd^2\widehat{H}_{\lambda}$ is circulant as well as symmetric, $\zeta^{(\ell)}_r$ and $\zeta^{(\ell)}_\theta$ (or rather their real and imaginary parts) are the eigenvectors.  The eigenvalues $\epsilon^{(\ell)}_r$ and $\epsilon^{(\ell)}_\theta$ (of course real) are found to be
\begin{equation}\label{eq:eigenvalues}
\begin{array}{ccl}
\epsilon^{(\ell)}_r &=& A + \frac{\kappa^2}{24\pi r_0^2}\left( (n^2-1)-6\ell(n-\ell)\right) \\[12pt]
 &=&  \frac{\kappa^2}{4\pi r_0^2}\left( 2(n-1)\frac{1+\lambda^2r_0^4}{\sigma^2} - \ell(n-\ell)\right)
\\[12pt]
\epsilon^{(\ell)}_\theta &=& \frac{\kappa^2}{4\pi}\ell(n-\ell),
\end{array} %Checked again (with r_0^{-1} etc)
\end{equation}
where $A$ is given in (\ref{eq:A}).  In case $\ell=1$, these simplify to
\begin{equation}\label{eq:evalues-rho1}
\epsilon^{(1)}_r = (n-1)\frac{\kappa^2\widetilde{\sigma}^{\, 2}}{4\pi r_0^2\sigma^2},\quad
\epsilon^{(1)}_\theta = (n-1)\frac{\kappa^2}{4\pi},
\end{equation}
which are both strictly positive.

%%%%%%%%%%%%

\subsection{Symplectic slice}\label{sec:symplectic slice}
Not all the eigenvalues are relevant to stability.  First, those that are zero because they correspond to directions along the group orbit should be discarded.  Second, those corresponding to directions transverse to the level set of the conserved quantities should be discarded, too. This is the process of reduction: restrict to ${\rm Ker}\, \dd\J$ (where $\mathbf{J=J}_\lambda$ is given in (\ref{eq:family of momentum maps})), then take the complement to the group orbit in that kernel.  The resulting space is called the \emph{symplectic slice}, which we denote by $\NN$.

We find ${\rm Ker}\,\dd\J$ using the Fourier basis above.  Identify $\gg_\lambda^*$ with $\CC\times\RR$.  Then

$$
\dd\J \zeta_r^{(0)} = \frac{nr_0}{\sigma^2}\pmatrix{0\cr 1}, \quad
\dd\J \zeta_r^{(1)} = \frac{n\widetilde{\sigma}}{\sigma^2}\pmatrix{1\cr 0}, \quad
\dd\J \zeta_\theta^{(1)} = \frac{n}{\:\sigma\,}\pmatrix{\ii\cr 0},
$$
while $\dd\J$ vanishes on all other Fourier tangent vectors. Write $V_1'=V_1\,\cap\, {\rm Ker}\,\dd\J$. Then
\begin{equation}\label{eq: V1'}
V_1'=\left<\sigma \alpha_r^{(1)} - \widetilde{\sigma}\beta_\theta^{(1)},\; \sigma \beta_r^{(1)}+\widetilde{\sigma}\alpha_\theta^{(1)}\right>.
\end{equation}
${\rm Ker}\,\dd\J$ is spanned by the $V_\ell\,$s with $\ell>1$, $V_1'$, and $\zeta_\theta^{(0)}$. The subspace generated by $\zeta_\theta^{(0)}$ is tangent to the group orbit (an infinitesimal rotation about the origin), so must be discarded.  Finally the symplectic slice is
\begin{equation}\label{eq:symplectic slice}
\NN=V_1'\oplus \bigoplus_{\ell=2}^{\lfloor n/2 \rfloor}V_\ell.
\end{equation}
The relevant eigenvalues are therefore $\epsilon^{(\ell)}_r$, $\epsilon^{(\ell)}_\theta$ for $\ell\geqslant 1$.  Now, with respect to the basis for $V_1'$ given in (\ref{eq: V1'}), the restriction of the Hessian to $V_1'$ is a scalar multiple of the identity.  So it has a double eigenvalue
\begin{equation}\label{eq:rho1'}
\epsilon_1' = \frac{n}2\,\sigma^2\epsilon_r^{(1)} + \frac{n}{2r_0^2}\widetilde{\sigma}^{\, 2}\epsilon_\theta^{(1)} =
\frac{n(n-1)\kappa^2}{4\pi r_0^2}\widetilde{\sigma}^{\, 2},
\end{equation}
which is strictly positive unless $\lambda r_0^2=1$. However, $\lambda r_0^2=1$ corresponds to the equator on the sphere, at which the momentum value is left fixed by all of $\SO(3)$.  In this case $\NN$ drops in dimension, $V_1'$ no longer lies in $\NN$, and the corresponding eigenvalue $\epsilon_1'$ becomes irrelevant to the stability of the relative equilibrium.

%%%%%%%%%%%%
\paragraph{Stability}
The relative equilibria in question are rotating rings; thus they are periodic trajectories. 
The precise sense of stability we are adopting is like the ordinary one of Lyapunov stability, but in terms of $G_\mu$-invariant open sets, where
$\mu$ is the momentum value at the trajectory and $G_\mu$ is the stabilizer of $\mu$: in detail, for every $G_\mu$-invariant neighbourhood $V$ of the trajectory, there exists a $G_\mu$-invariant neighbourhood $U \subseteq V$ such that every trajectory starting in $U$ remains in $V$ for all time.

We can make do with a coarse criterion: if the restriction  of the Hessian of the augmented Hamiltonian to the symplectic slice is positive-definite, then the relative equilibrium is stable in our sense.  A finer criterion is: if this Hessian is merely non-negative but the augmented Hamiltonian admits a local extremum at the relative equilibrium, then the relative equilibrium is still stable \cite[Theorem 1.2]{Mo97}.

Among the relevant eigenvalues, $\epsilon_\theta^{(\ell)} > 0$ for $\ell \geqslant 1$ and 
$\epsilon_1'> 0$.  It remains to check the sign of $\epsilon_r^{(\ell)}$ for $\ell\geqslant 2$.
From the expression in (\ref{eq:eigenvalues}), it is clear that the least eigenvalue occurs for $\ell=\lfloor n/2 \rfloor$.  This is the criterion we have been after: the relevant eigenvalues are all strictly positive if and only if
\begin{equation}\label{eq:stability criterion}
 \frac{1+\lambda^2r_0^4}{(1+\lambda r_0^2)^2} > \frac1{2(n-1)}\Big\lfloor\frac{n^2}{4}\Big\rfloor .
\end{equation}
The left-hand side is unbounded as a function of $r_0$ if $\lambda < 0$.  
The right-hand side is a strictly increasing, unbounded function of $n$ for $n\geqslant 3$.  Hence, on one hand, for fixed $\lambda$, $r_0$ a value of $n$ exists beyond which the inequality fails; on the other hand, if $\lambda < 0$, for fixed $n$, by enlarging $r_0^2$ sufficiently close to $-1/\lambda$ (its supremum on the hyperbolic plane), we can ensure the inequality holds. We 
conclude with a result which for spheres is due to \cite{BC03} (see also \cite{LMR}):
\begin{theorem}
  On $\M_\lambda$, the relative equilibrium of $n$ identical point vortices in a regular ring of radius $a$ is Lyapunov-stable if (\ref{eq:stability criterion}) is satisfied, where $a$ and $r_0$ are related by (\ref{eq:radius}).  In particular, in the hyperbolic scenario $\lambda < 0$ this ring is stable if\/ $\lambda r_0^2$ is sufficiently close to $-1$.
\end{theorem}

\begin{remark} 
Physical intuition confirms that a spherical surface destabilizes a ring whereas a hyperbolic surface stabilizes it.  Think of a vortex of the ring. As
$\lambda$ gets positive, the adjacent vortices remain relatively near while the diametrically opposite vortices become relatively far, so the former, which tend to knock our vortex perpendicularly to the ring, exert more influence than the latter, which tend to slide our vortex tangentially to the ring.  As
$\lambda$ gets negative, the effects are felt the other way round.
\end{remark}

\begin{remark}
If the inequality in (\ref{eq:stability criterion}) is reversed, then the ring is \emph{linearly unstable}.  This is because the $V_{\lfloor n/2\rfloor}$ mode will have real eigenvalues.  Indeed, for $n$ even $V_{n/2}$ is of dimension 2 and the Hessian is indefinite, which suffice to conclude that the eigenvalues are real.  For $n$ odd $V_{(n-1)/2}$ is 4-dimensional, so having an indefinite Hamiltonian does not imply the eigenvalues are real.  However, the negative eigenspace is spanned by $\alpha_r^{(\ell)},\beta_r^{(\ell)}$ which is Lagrangian, and then the realness of the eigenvalues follows.

If the parameters are such that (\ref{eq:stability criterion}) is an equality, then the stability is not determined by linear analysis.   This determination is what our deformation plus bifurcation approach achieves, in sections \ref{sec:n even} and \ref{sec:n odd}.
\end{remark}

We now spell out the conclusions concretely for each of the three geometries. See Figure \ref{fig:stability ranges} and the table of bifurcation points below. On the spheres $\lambda>0$ there are two bifurcation points: the value of $\lambda r_0^2$ listed in the table below and its reciprocal.

$$\begin{array}{c|cccccccccccccc}
n& 4&5&6&7&8&9&10&11&12&13\\
\hline
\rule{0pt}{2.3ex}\;\lambda r_0^2 &0.268 &0.172&0.0557 &0&-0.0627&-0.101&-0.143 & -0.172& -0.202&-0.225
\end{array}
$$

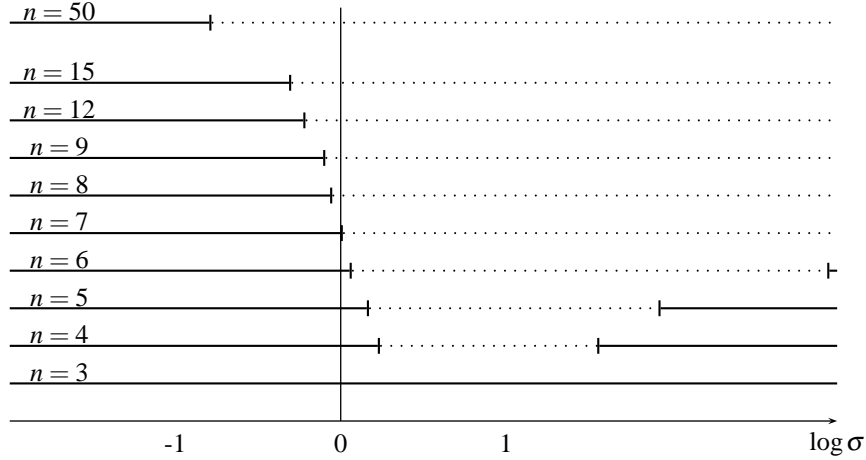
\begin{figure}[t]
\begin{center}
\psset{xunit=2.2}%2.75}
\begin{pspicture}(-2,-2)(3.1,4)
  \psline[linewidth=0.5pt]{->}(-2,-1.5)(3,-1.5)   \rput(3,-1.8){$\log\sigma$}
  \rput(-1,-1.8){-1}   \rput(0,-1.8){0}  \rput(1,-1.8){1}
  \psline[linewidth=0.5pt](0,-1.5)(0,4) % vertical axis
  \rput(-1.7,3.95){$n=50$} \psline(-2,3.8)(-0.794,3.8)
     \psline[linestyle=dotted]{|-}(-0.794,3.8)(3,3.8)
  \rput(-1.7,3.15){$n=15$} \psline(-2,3)(-0.312,3)
     \psline[linestyle=dotted]{|-}(-0.312,3)(3,3)
  \rput(-1.7,2.65){$n=12$} \psline(-2,2.5)(-0.226,2.5)
     \psline[linestyle=dotted]{|-}(-0.226,2.5)(3,2.5)
  \rput(-1.7,2.15){$n=9$} \psline(-2,2)(-0.106,2)
     \psline[linestyle=dotted]{|-}(-0.106,2)(3,2)
  \rput(-1.7,1.65){$n=8$} \psline(-2,1.5)(-0.064,1.5)
     \psline[linestyle=dotted]{|-}(-0.064,1.5)(3,1.5)
  \rput(-1.7,1.15){$n=7$} \psline(-2,1)(0,1)
     \psline[linestyle=dotted]{|-}(0,1)(3,1)
  \rput(-1.7,0.65){$n=6$} \psline(-2,0.5)(0.054,0.5)
     \psline[linestyle=dotted]{|-}(0.054,0.5)(2.94,0.5)
     \psline{|-}(2.94,0.5)(3,0.5)
  \rput(-1.7,0.15){$n=5$} \psline(-2,0)(0.158,0)
     \psline[linestyle=dotted]{|-}(0.158,0)(1.92,0)
     \psline{|-}(1.92,0)(3,0)
  \rput(-1.7,-0.35){$n=4$} \psline{-|}(-2,-0.5)(0.237,-0.5)
     \psline[linestyle=dotted](0.237,-0.5)(1.55,-0.5)
     \psline{|-}(1.55,-0.5)(3,-0.5)
  \rput(-1.7,-0.85){$n=3$} \psline(-2,-1)(3,-1)
\end{pspicture}
\caption{Ranges of stability (solid lines) and instability (dotted lines) for the dimensionless quantity $\log\sigma = \log(1+\lambda r_0^2)$}
\label{fig:stability ranges}
\end{center}
\end{figure}

\paragraph{Plane} There is an obvious scale-invariance, and the stability/instability of the ring is independent of the radius $a$. We recover J.J.~Thomson's original result \cite{Thomson} that the ring is stable when $n<7$ and unstable when $n>7$. When $n=7$, we have the so-called Thomson heptagon, whose stability is not determined by linear analysis (but cf.~Corollary \ref{coroll:Thomson}). 

\paragraph{Spheres}  When $n\geqslant 7$, the ring is always unstable. When $n<7$, there is a range of $a$ over which it is stable.  Example: when $n=6$, it is stable for $(\lambda r_0^2-2)^2>3$ which translates, bearing in mind $\lambda>0$, into
$$ 
r_0^2 < \frac{2-\sqrt{3}}{\lambda} \quad\mbox{or}\quad r_0^2 > \frac{2+\sqrt{3}}{\lambda}.
$$
The two inequalities correspond to neighbourhoods of the South and North Poles respectively. For $n=3$ the ring is always stable, at any radius $a$.  The linear stability of rings of vortices on the sphere was first studied by Polvani and Dritschel \cite{PD93}, and the full nonlinear stability in \cite{BC03}---see also \cite{LMR} for more details.

\paragraph{Hyperbolic planes}  When $n\leqslant 7$, the ring is always stable. When $n>7$, it is stable (for a given $\lambda$) for $a$ sufficiently large.  Example: when $n=15$, it is stable for
$$
r_0^2 > \frac{2-\sqrt{3}}{|\lambda|}.
$$

%%%%%%%%%%%%%%%%%%%%%%%%%%%%%%%%%%%%%
\section{Bifurcations across the degeneracy}
\label{sec:bifurcations}

To understand the bifurcations in detail, it is imperative to exploit the symmetries of the system, which for the ring of point vortices is the dihedral group.  We begin section \ref{sec:dihedral action}
 by setting up the dihedral symmetry of the system, and then in \ref{sec:dihedral bifurcations} list the bifurcations we expect in Hamiltonian systems with dihedral symmetry.  Following that, in \ref{sec:bifurcations of rings}, we state the main theorem (Theorem
 \ref{thm:bifurcations}) on which of these bifurcations occur in the dynamics of point vortices;  
 it leads to the nonlinear stability of the Thomson heptagon (Corollary \ref{coroll:Thomson}).  In \ref{sec:geometry of bifurcating rings} we visit the geometry of the bifurcating rings, which enjoy less symmetry (spontaneous symmetry breaking) and are illustrated in Figure \ref{fig:perturbations}. Finally, \ref{sec:degenerate} proves the main theorem: we perform the calculations needed to justify these `expectations', and decide which of the expectations are actually  realized.  For the sake of completeness, 
 section \ref{sec:equator} summarizes results from \cite{LMR01} on bifurcations from the equator; they are not covered by the other results as the momentum value there is degenerate (in the sense that it is fixed by the entire group rather than just a 1-parameter subgroup).

\subsection{Dihedral group action}\label{sec:dihedral action}

For the simplicity of language, we shall confuse $\M_\lambda$ and the uniform chart 
$\CC$ of section \ref{sec:surfaces}.  Points and group actions on $\CC$ should be interpreted as
their lifts on $\M_\lambda$.

The full system is invariant under the symmetry group $G_\lambda$ (depending  on $\lambda$) as in section \ref{sec:Lie algebra}. For every $\lambda$, $G_\lambda$ contains rotations about the origin, and reflections in lines through the origin, together generating a subgroup of $G_\lambda$ isomorphic to $\OO(2)$.  Consider now a system of $n$ identical point vortices.  A dihedral subgroup\footnote{$S_n$ is the group of permutations of the $n$ vortices and $\D_n$ has order $2n$.}
 $\D_n \subseteq \OO(2)\times S_n$ acts on the phase space $\CC^n$ by
\begin{equation}\label{eq:D-n action on phase space}
\begin{array}{lcl}
c \cdot (z_1,\dots,z_n) &=& (c z_n, c z_1,\dots, c z_{n-1})\\
m \cdot (z_1,\dots,z_n) &=& (\bar z_{n-1}, \bar z_{n-2},\dots, \bar z_1,\bar z_n).
\end{array}
\end{equation}
where $c=\exp(2\pi\ii/n)$ is a cyclic rotation and $m$ is a mirror reflection.  If the points 
$z_j=r_0\exp(2\pi\ii j/n)$ are placed as a regular ring, then that configuration is fixed by this
$\D_n$, and $m$ acts as a reflection in the line passing through $z_n$.  In case $n$ is odd, all reflections in $\D_n$ are conjugate, whereas in case $n$ is even, there are two distinct conjugacy classes of reflections, one consisting of those through opposite vertices of a regular $n$-gon (all conjugate to 
$m$), the other consisting of those through mid-points of opposite edges (all conjugate to $m'= c \, m$). This will come into play in deciding what bifurcating solutions appear.

Since the $\D_n$-action fixes the ring, $\D_n$ acts on the tangent space to the phase space at this ring. The Fourier basis (\ref{eq:Fourier basis}) for $n$ point vortices is adapted to this action: 
\begin{equation}\label{eq:action on modes}
\begin{array}{ccccccc}
c\cdot\zeta_r^{(\ell)} &=& \ee^{-2\pi\ii\ell/n}\,\zeta_r^{(\ell)} &\quad& m\cdot\zeta_r^{(\ell)} &=& \bar\zeta_r^{(\ell)}\\
c\cdot\zeta_\theta^{(\ell)} &=& \ee^{-2\pi\ii\ell/n}\,\zeta_\theta^{(\ell)} && m \cdot\zeta_\theta^{(\ell)} &=&  - \bar\zeta_\theta^{(\ell)},
\end{array}
\end{equation}
and $m'\cdot\zeta_r^{(\ell)} = \ee^{-2\pi\ii\ell/n} \,\bar\zeta_r^{(\ell)}$,
$m'\cdot\zeta_\theta^{(\ell)} = -\ee^{-2\pi\ii\ell/n}\, \bar\zeta_\theta^{(\ell)}$.
For the 2-dimensional subspace $V_1'$ of the symplectic slice $\NN$ (\ref{eq: V1'}), write
\begin{equation}\label{eq:zeta'}
\zeta' = (\sigma \alpha_r^{(1)} - \widetilde{\sigma}\beta_\theta^{(1)}) + \ii(\sigma\beta_r^{(1)} + 
\widetilde{\sigma}\alpha_\theta^{(1)}) = \sigma\zeta_r^{(1)} + \ii\widetilde{\sigma}\zeta_\theta^{(1)}.
\end{equation}
Then (\ref{eq:action on modes}) implies that $c\cdot\zeta' = \ee^{-2\pi\ii/n}\,\zeta'$, 
$m\cdot\zeta'= \overline{\zeta'}$, and $m'\cdot \zeta' = \ee^{-2\pi\ii/n}\,\overline{\zeta'}$.

\medskip

As the parameter $\lambda$ varies, the eigenvalue $\epsilon_r^{(\ell)}$, $\ell \geqslant 2$ in (\ref{eq:eigenvalues}) may cross $0$ and may involve a bifurcation in the mode $V_\ell$. In contrast  $\epsilon_r^{(1)}$ and $\epsilon_\theta^{(\ell)}$, being strictly positive, never involve bifurcations and in particular the mode $V_1'$ never bifurcates.

%%%%%%%%%%%%%%%%%%%%%%

\subsection{Dihedral bifurcations}\label{sec:dihedral bifurcations}

The type of bifurcation expected in a symmetric Hamiltonian system is controlled by the group action on the generalized kernel of the linear system at the bifurcation point.  This was first investigated by Golubitsky and Stewart \cite{GS87}, cf.\ also \cite{BLM05}.  It follows from \cite{GS87} that in a {\it generic\/} family of linear Hamiltonian systems, a pair of eigenvalues come together along the imaginary axis, collide at the origin, and split along the real axis. This splitting transition is indeed what happens in our problem, as seen from the expressions (\ref{eq:eigenvalues}) for the eigenvalues. Part of the {\it genericity\/} hypothesis of \cite{GS87} is that the generalized kernel be an irreducible symplectic representation, which is satisfied here.

The greatest common divisor of $n$ and $\ell$ will be denoted by $(n, \ell)$.  We see  from (\ref{eq:action on modes}) that the cyclic subgroup $\ZZ_{(n,\ell)} \subset \D_n$ acts trivially on $V_{\ell}$.  Consequently on $V_\ell$ there is an effective action of $\D_{n/(n,\ell)}$. It turns out that $V_\ell$ is an irreducible symplectic representation of this group.  Two cases are to be distinguished: $\ell=n/2$ ($n$ even) when $V_{n/2}$ is 2-dimensional with an action of $\D_2 =\ZZ_2\times\ZZ_2$, and $\ell\neq n/2$ when the $V_\ell\,$s for $0<\ell<n/2$ are all 4-dimensional. For bifurcations the modes $1<\ell\leqslant n/2$ alone are of interest to us.
Much of the analysis of generic bifurcations with dihedral symmetry is found in \cite{GSS88}, and although there they deal with general vector fields rather than with Hamiltonian ones, the conclusions turn out to be the same. Analysis of the gradient case is also in \cite{BF93}.

\bigskip

In the dynamics of point vortices, if the bifurcating mode is $\ell\neq \lfloor n/2 \rfloor$, then the linear system\footnote{The vector field, not the Hessian.} at the relative equilibrium has real eigenvalues in the $\lfloor n/2 \rfloor$ mode,
hence is unstable.  If $\ell = \lfloor n/2 \rfloor$, then the bifurcation involves a loss of stability in the mode $\ell$.  This means that in our analysis below, a local minimum corresponds to stable relative equilibria only if we are looking at the $\lfloor n/2 \rfloor$ mode.

\paragraph{Bifurcations on $V_{\ell}$ for $\ell = n/2$}
This is the 2-dimensional symplectic span 
$$\left< \zeta^{(n/2)}_r, \zeta^{(n/2)}_\theta\right>,$$ 
with an action of $\D_2\simeq\ZZ_2\times\ZZ_2$.  The kernel of the Hessian at the bifurcation point is the 1-dimensional subspace spanned by $\zeta^{(n/2)}_r$, with a $\ZZ_2$-action. Then the generic bifurcation is a $\ZZ_2$-pitchfork, which can be either sub- or super-critical: if
subcritical, the bifurcating solutions are unstable, and if supercritical, they are stable (provided the original `central' solution was stable). A normal form is given by the family
\begin{equation}\label{eq:generic D2 function}
f_u(x,y) = - ux^2 \pm x^4 + y^2 + \mbox{h.o.t.}
\end{equation}
`h.o.t.' stands for higher-order terms in $x^2, y^2, u$.  The $+$ sign in front of $y^2$ is justified by the eigenvalue in the $\zeta^{(n/2)}_\theta$-direction, which is always positive.  The $-$ sign in
front of $u$ is a choice, dictated by the fact that increasing $\lambda r_0^2$ makes a critical point pass from local minimum to saddle, as shown in Figure \ref{fig:stability ranges}. The sign $+$ or $-$ in front of $x^4$ corresponds to supercritical or subcritical, respectively. See Figure \ref{fig:dihedral bifurcation diagrams}(a,b) and the lecture notes \cite{BLM05} for a fuller discussion. In Remark~\ref{rk:supercritical} we explain why the bifurcations occurring here are in fact all supercritical, for all even $n$.

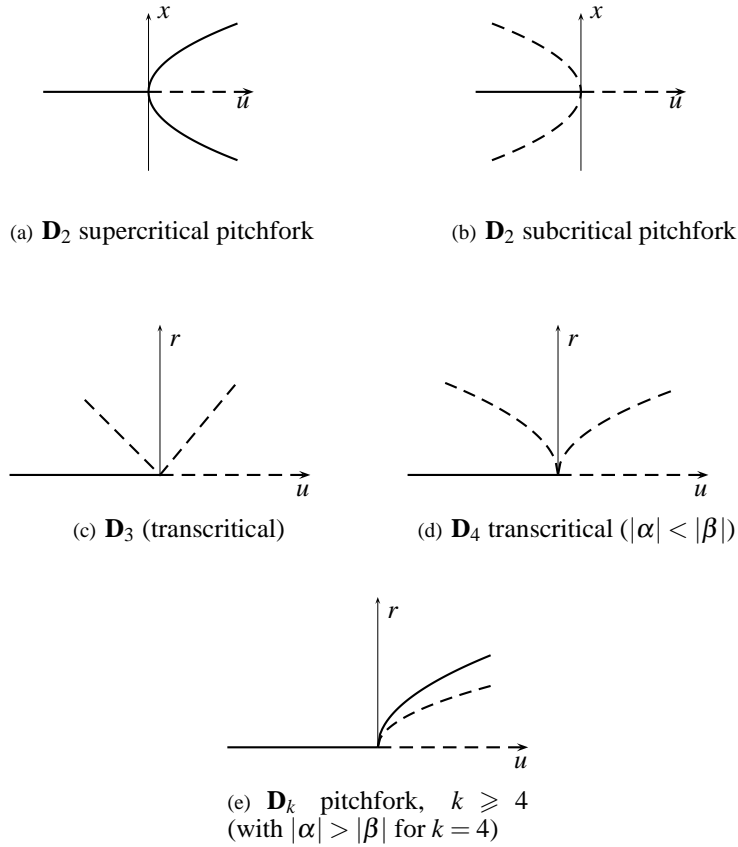
\begin{figure}[t]
\begin{center}\psset{unit=0.7}
\subfigure[\normalsize$\D_2$ supercritical pitchfork]{%
\begin{pspicture}(-3,-2.2)(3.5,1.5)
\psline(-2,0)(0,0)
\psline[linestyle=dashed]{->}(0,0)(2,0)\rput(1.8,-0.2){$u$}
\psline[linewidth=0.3pt]{->}(0,-1.5)(0,1.5)
\parametricplot{-1.3}{1.3}{t 2 exp t}
\rput(0.3,1.5){$x$}
\end{pspicture}}\hfil
\subfigure[\normalsize$\D_2$ subcritical pitchfork]{%
\begin{pspicture}(-3,-2.2)(3.5,1.5)
\psline(-2,0)(0,0)
\psline[linestyle=dashed]{->}(0,0)(2,0)\rput(1.8,-0.2){$u$}
\psline[linewidth=0.3pt]{->}(0,-1.5)(0,1.5)
\parametricplot[linestyle=dashed]{-1.3}{1.3}{t 2 exp neg t}
\rput(0.3,1.5){$x$}
\end{pspicture}}
\end{center}
\begin{center}
\subfigure[\normalsize$\D_3$ (transcritical)]{%
\begin{pspicture}(-2,-0.4)(2.5,2.5)
\psline(-2,0)(0,0)
\psline[linestyle=dashed]{->}(0,0)(2,0)
\psline[linewidth=0.2pt]{->}(0,0)(0,2)
\psline[linestyle=dashed](-1,1)(0,0)(1,1.2)
\rput(0.2,1.8){$r$}
\rput(1.9,-0.2){$u$}
\end{pspicture}}
  \qquad
\subfigure[\normalsize$\D_4$ transcritical ($|\alpha|<|\beta|$)]{%
\begin{pspicture}(-2,-0.4)(2.5,2.5)
\psline(-2,0)(0,0)
\psline[linestyle=dashed]{->}(0,0)(2,0)
\psline[linewidth=0.2pt]{->}(0,0)(0,2)
\psplot[linestyle=dashed]{-1.5}{0}{x -1 mul sqrt}
\psplot[linestyle=dashed]{0}{1.5}{x sqrt 1.1 div}
\rput(0.2,1.8){$r$}
\rput(1.9,-0.2){$u$}
\end{pspicture}}
  \qquad
\subfigure[\normalsize$\D_k$ pitchfork, $k\geqslant 4$ \quad (with $|\alpha|>|\beta|$ for $k=4$)]{%
\begin{pspicture}(-2,-0.4)(2,2.5)
\psline(-2,0)(0,0)
\psline[linestyle=dashed]{->}(0,0)(2,0)
\psline[linewidth=0.2pt]{->}(0,0)(0,2)
\psplot{0}{1.5}{x sqrt}
\psplot[linestyle=dashed]{0}{1.5}{x sqrt 1.5 div}
\rput(0.2,1.8){$r$}
\rput(1.9,-0.2){$u$}
\end{pspicture}}
\end{center}
\caption{Bifurcation diagrams for generic $\D_k$-bifurcations, $k\geqslant 2$. $r=\sqrt{x^2+y^2}$, and $u$ is the parameter as in the text. Solid lines refer to local minima, dashed lines to saddles and local maxima. For $k\geqslant 3$ each nontrivial branch corresponds to $k$ solutions after applying the rotations from $\D_k$. Figure (e) is drawn for $\alpha>1$, where $\alpha$ is the coefficient in (\ref{eq:generic Dn function}); if $\alpha <-1$, reflect the diagram in the $r$-axis.}
\label{fig:dihedral bifurcation diagrams}
\end{figure}

\paragraph{Bifurcations on $V_\ell$ for $\ell\neq n/2$} 
Of these only $\ell=\half(n-1)$ involves stable relative equilibria, other modes bifurcate only if the linear system already has real eigenvalues; that said, the other values of $\ell$ do involve the appearance of new unstable relative equilibria, so are of interest.  Write $k=n/(n,\,\ell)$.  Then $k \geqslant 4$ and $\D_k$ acts
effectively on $V_\ell$. The type of generic bifurcation we get depends on $k$. The 4-dimensional $V_\ell$ is a direct sum of two $\D_k$-invariant 2-dimensional Lagrangian subspaces, on one of which the Hessian vanishes at the bifurcation point, on the other it is always positive-definite.  A $\D_k$-invariant function on such a space is a function of the fundamental invariants
$$
N(x,y) = x^2+y^2,\quad P(x,y) = \mathop{\textrm{Re}}(x+\ii y)^k.
$$
A generic 1-parameter family of such functions is given by
\begin{equation}\label{eq:generic Dn function}
f_u = -uN + \alpha N^2+\beta P + \mbox{h.o.t.},
\end{equation}
where `h.o.t.' stands for higher-order terms in $N$, $P$, $u$.
Figure \ref{fig:dihedral bifurcations} shows the level sets of $f_u$ for $k=3, \ldots, 6$, as $u$ varies through $0$.

\begin{figure}[t]
\begin{center}

\subfigure[$\D_3$ with $u<0$]{\includegraphics[scale=0.15]{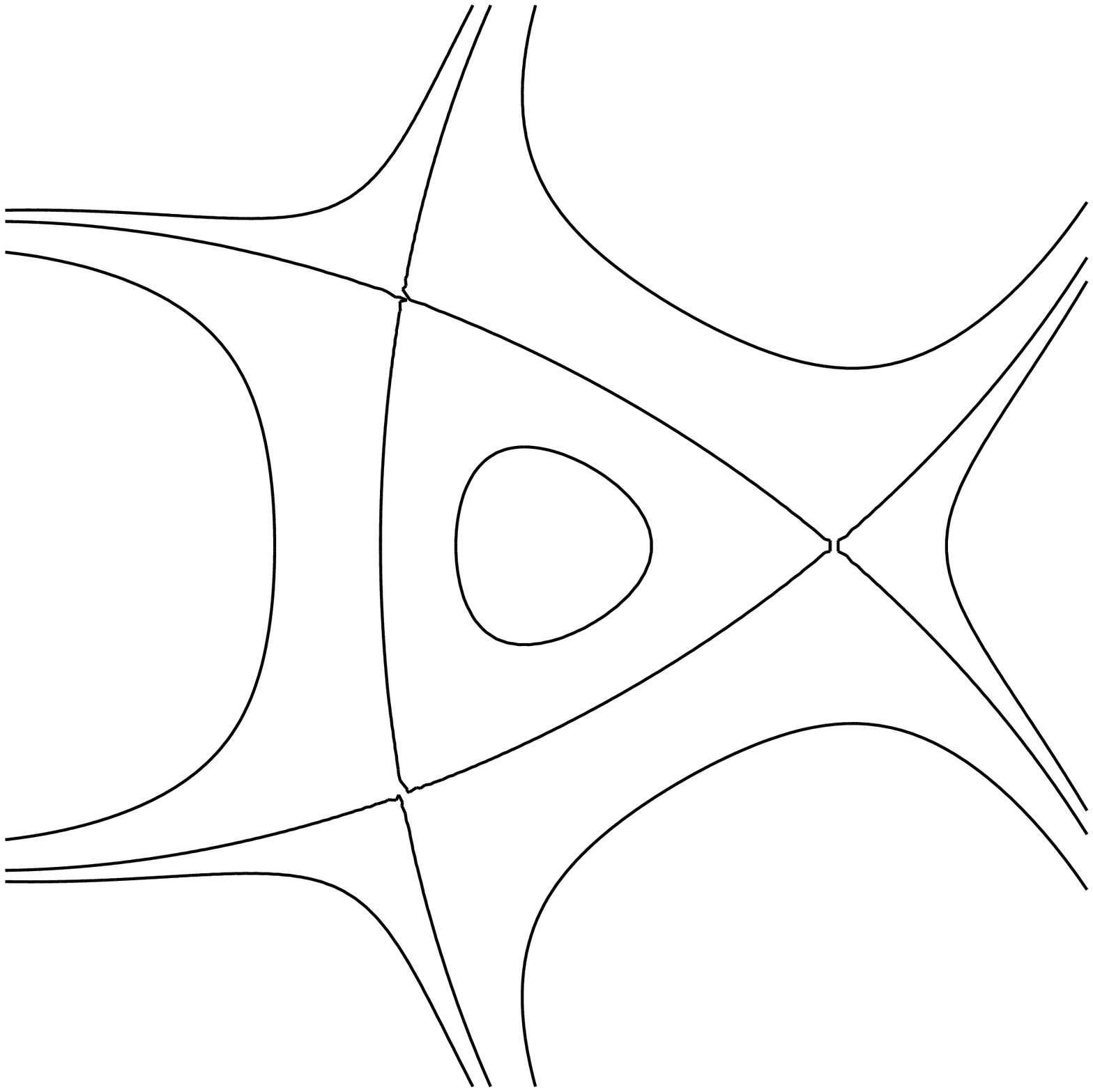}}
 \qquad\qquad
\subfigure[$\D_3$ with $u>0$]{\includegraphics[scale=0.15]{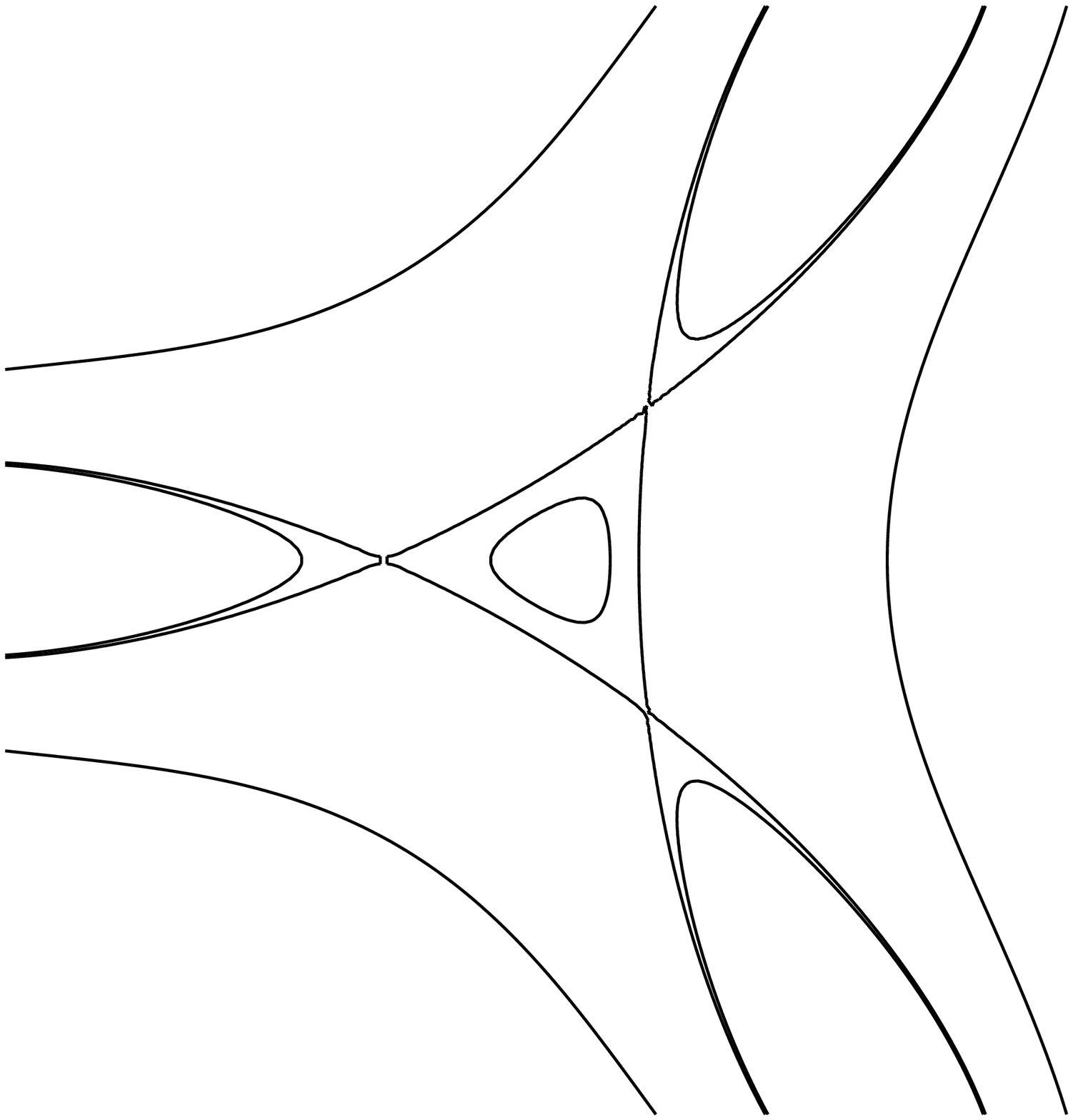}}

\subfigure[$\D_4$ with $|\alpha|<|\beta|$ and $u<0$]{\includegraphics[scale=0.15]{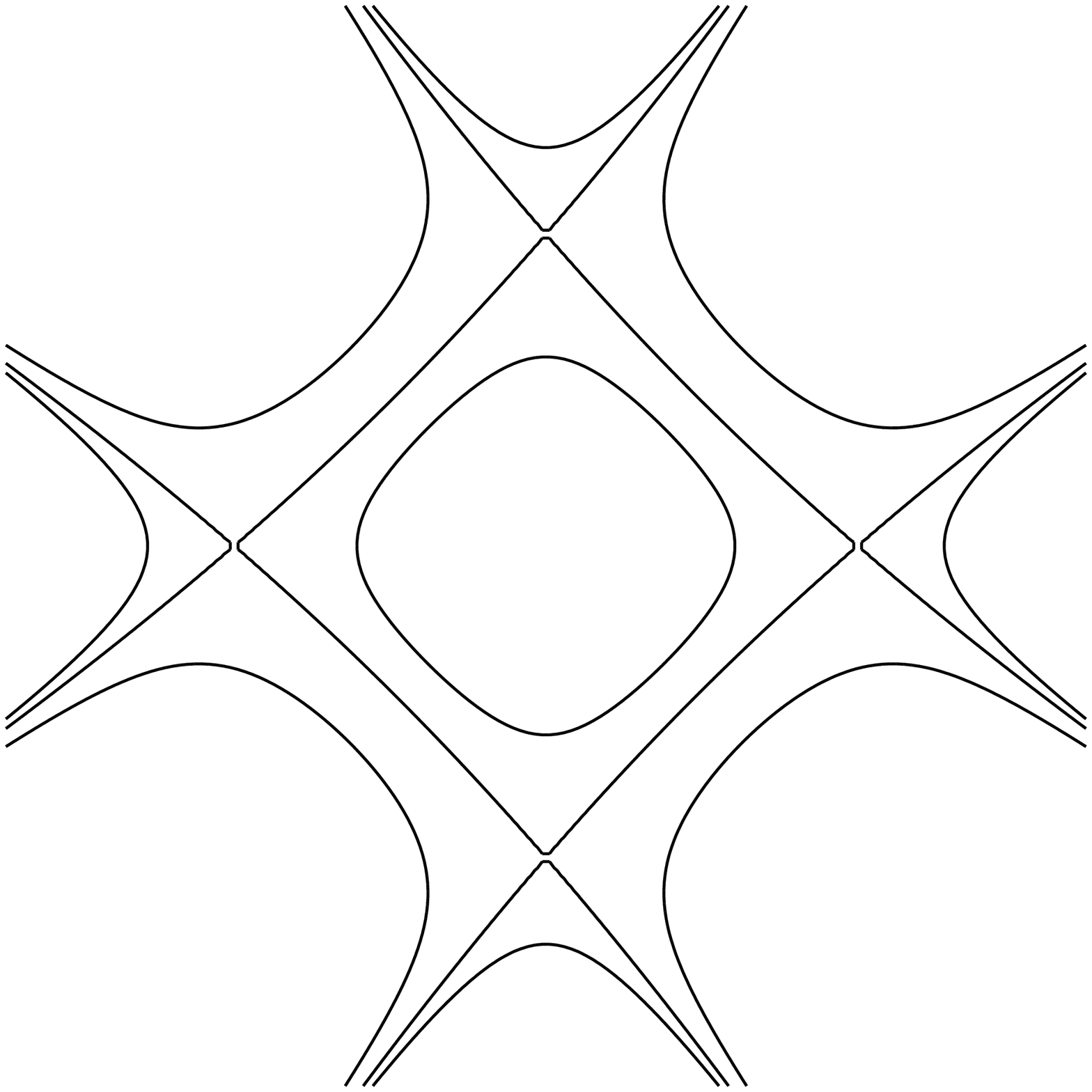}}
 \qquad\qquad
\subfigure[$\D_4$ with $|\alpha|<|\beta|$ and $u>0$]{\includegraphics[scale=0.15]{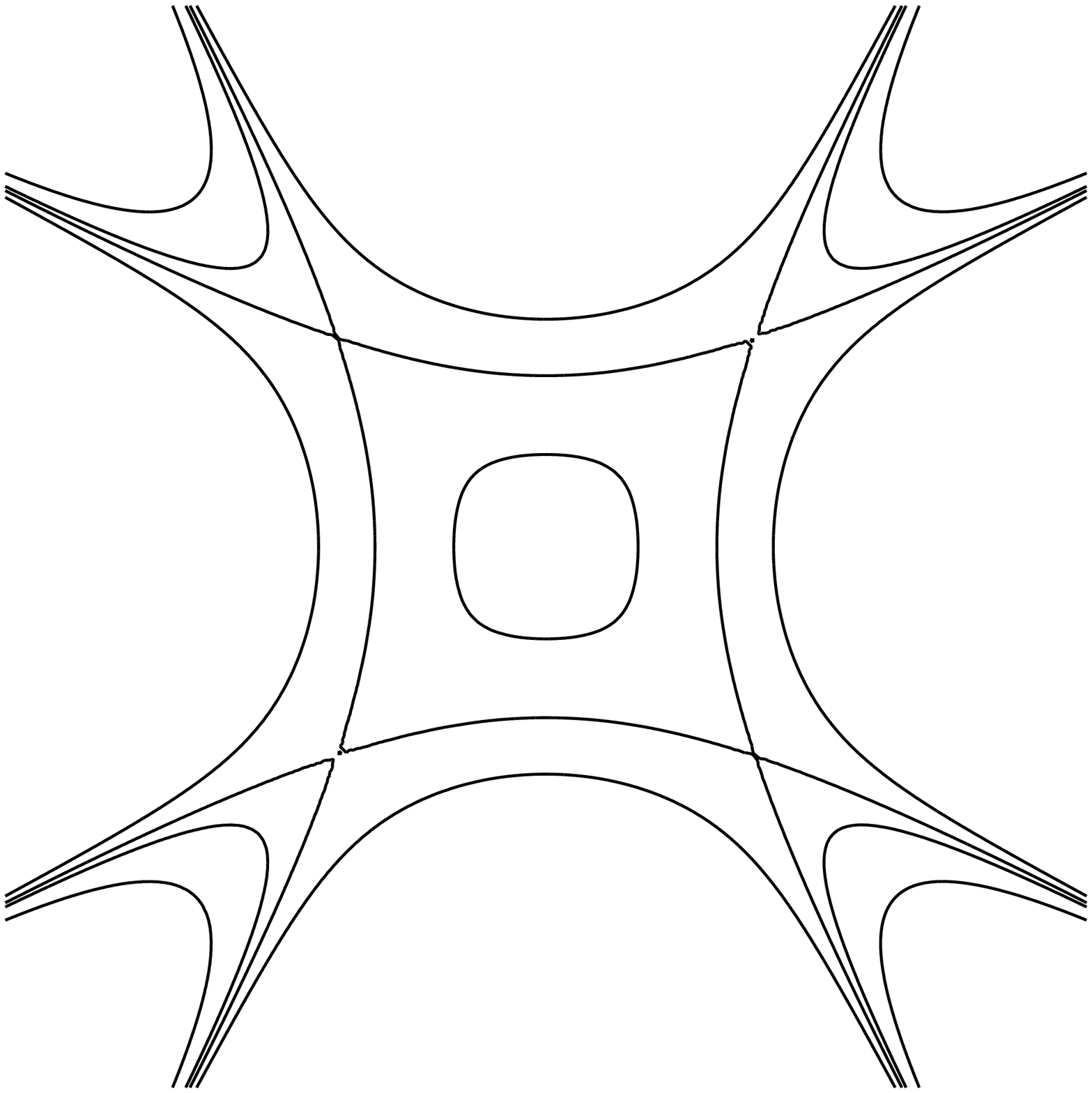}}

\caption{Contours of the generic 1-parameter family of $\D_k$-invariant functions (\ref{eq:generic Dn function}), for $k = 3,\,4$ [produced with {\tt Maple}, with a judicious choice of level sets] . The figures (a)--(d) are all transcritical bifurcations.  Continued on next page.}
\label{fig:dihedral bifurcations}
\end{center}
\end{figure}

\begin{figure}[t]
\addtocounter{figure}{-1}%\samenumber
\setcounter{subfigure}{4}
\begin{center}
%\bigskip

%
\subfigure[$\D_4$ with $|\alpha|>|\beta|$, $u<0$]{\includegraphics[scale=0.15]{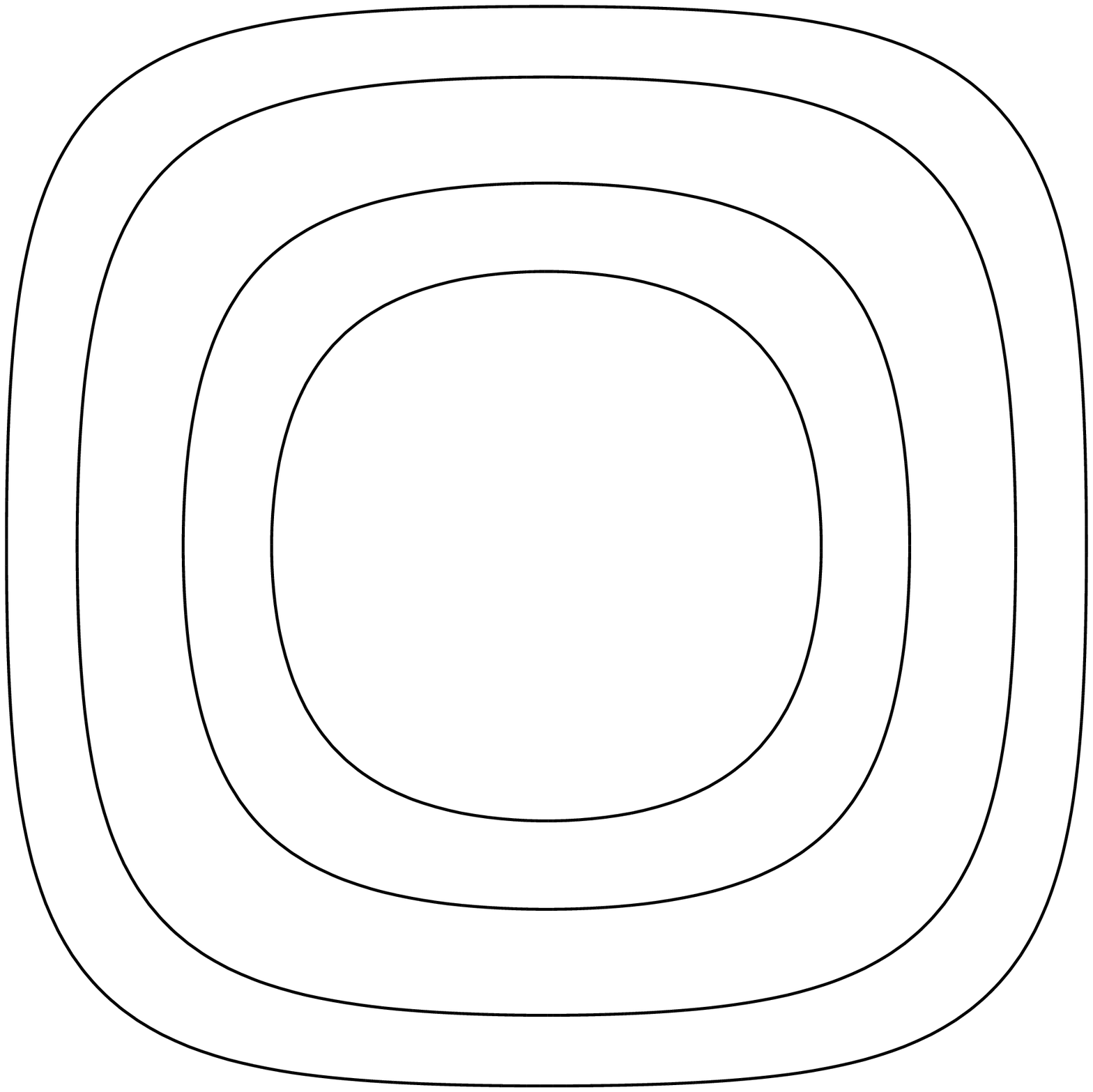}}
 \qquad\qquad
\subfigure[$\D_4$ with $|\alpha|>|\beta|$, $u>0$]{\includegraphics[scale=0.15]{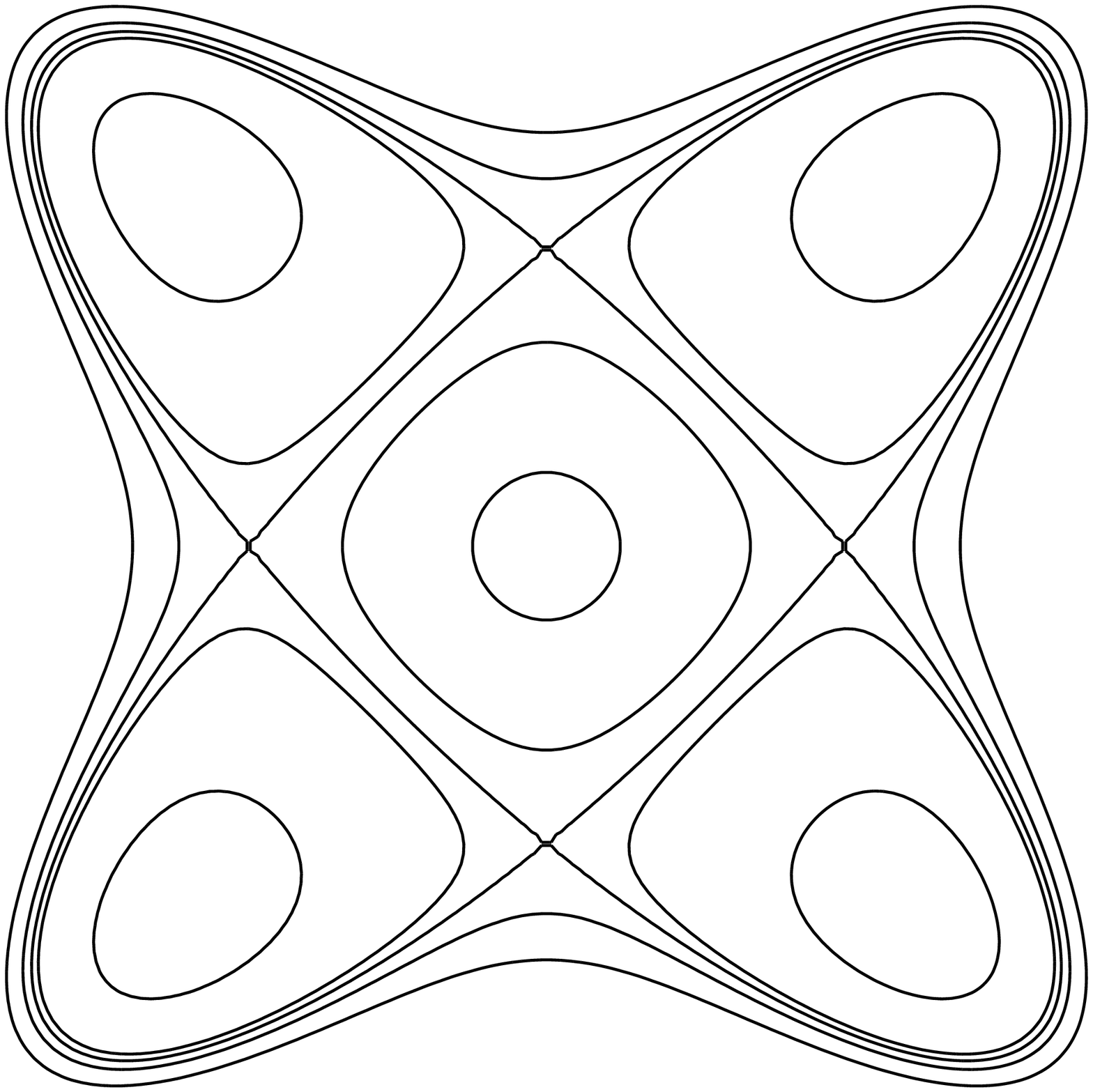}}
\subfigure[$\D_5$ with $u<0$]{\includegraphics[scale=0.15]{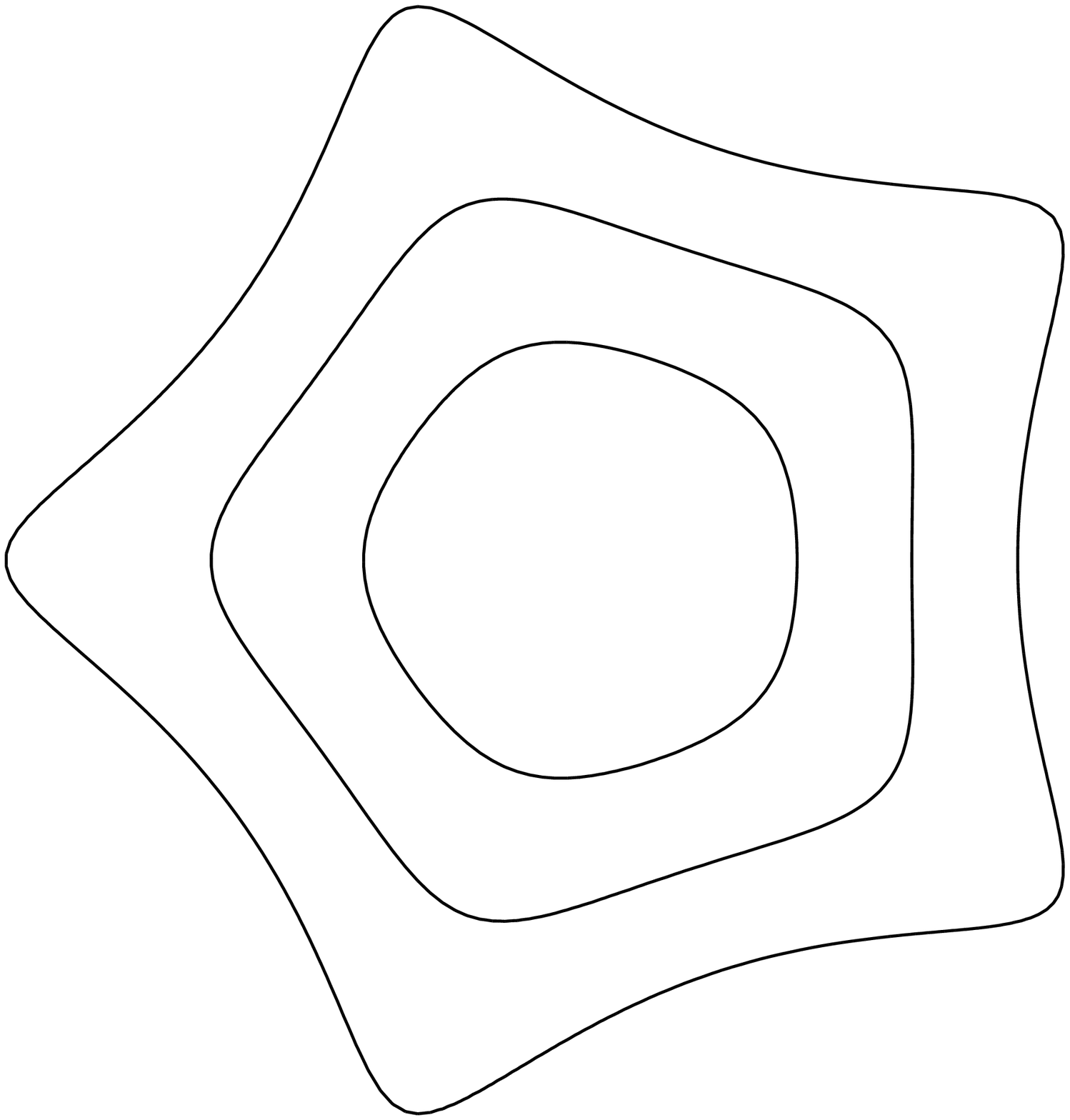}}
 \qquad\qquad
\subfigure[$\D_5$ with $u>0$]{\includegraphics[scale=0.15]{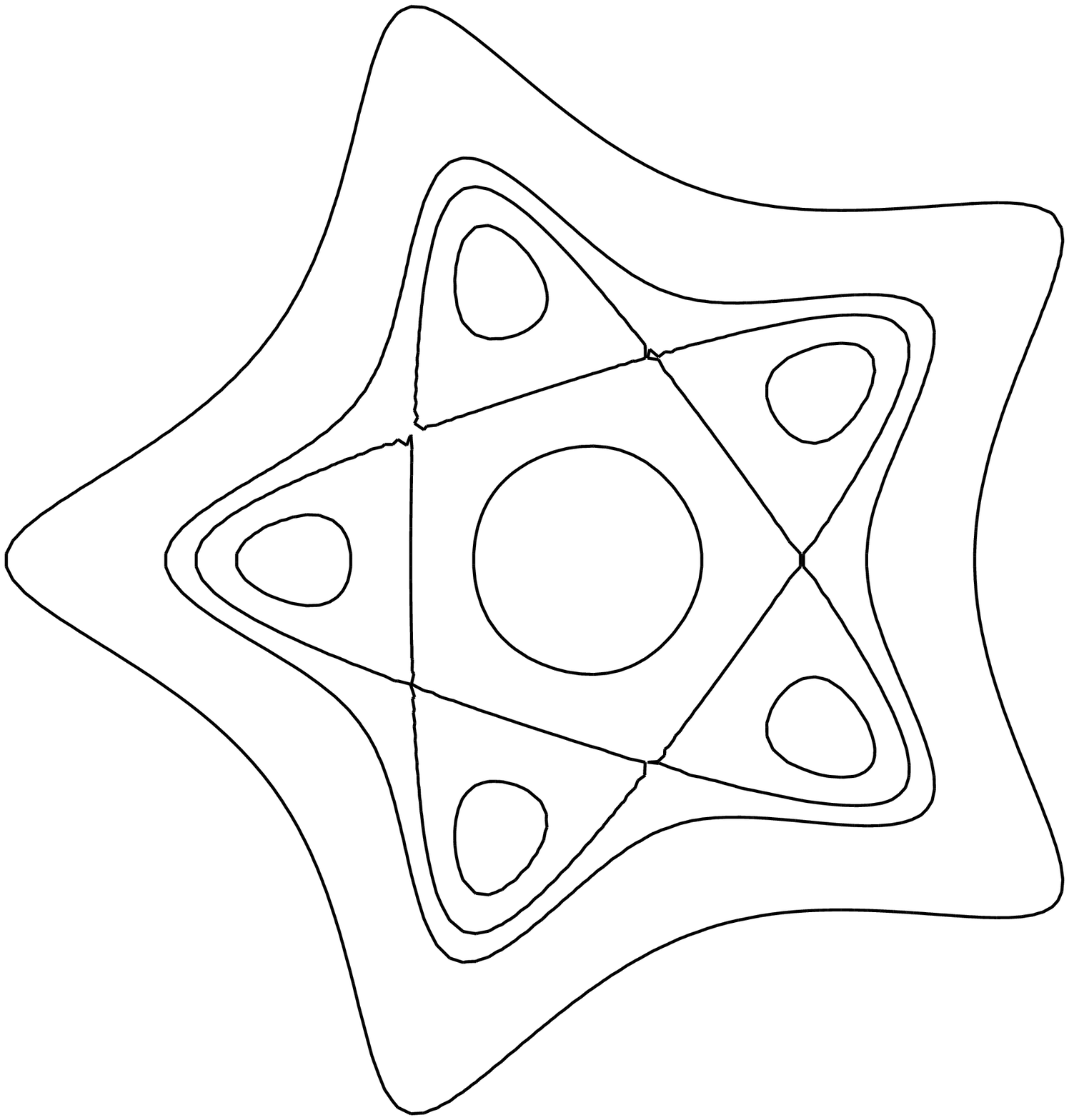}}
\subfigure[$\D_6$ with $u<0$]{\includegraphics[scale=0.15]{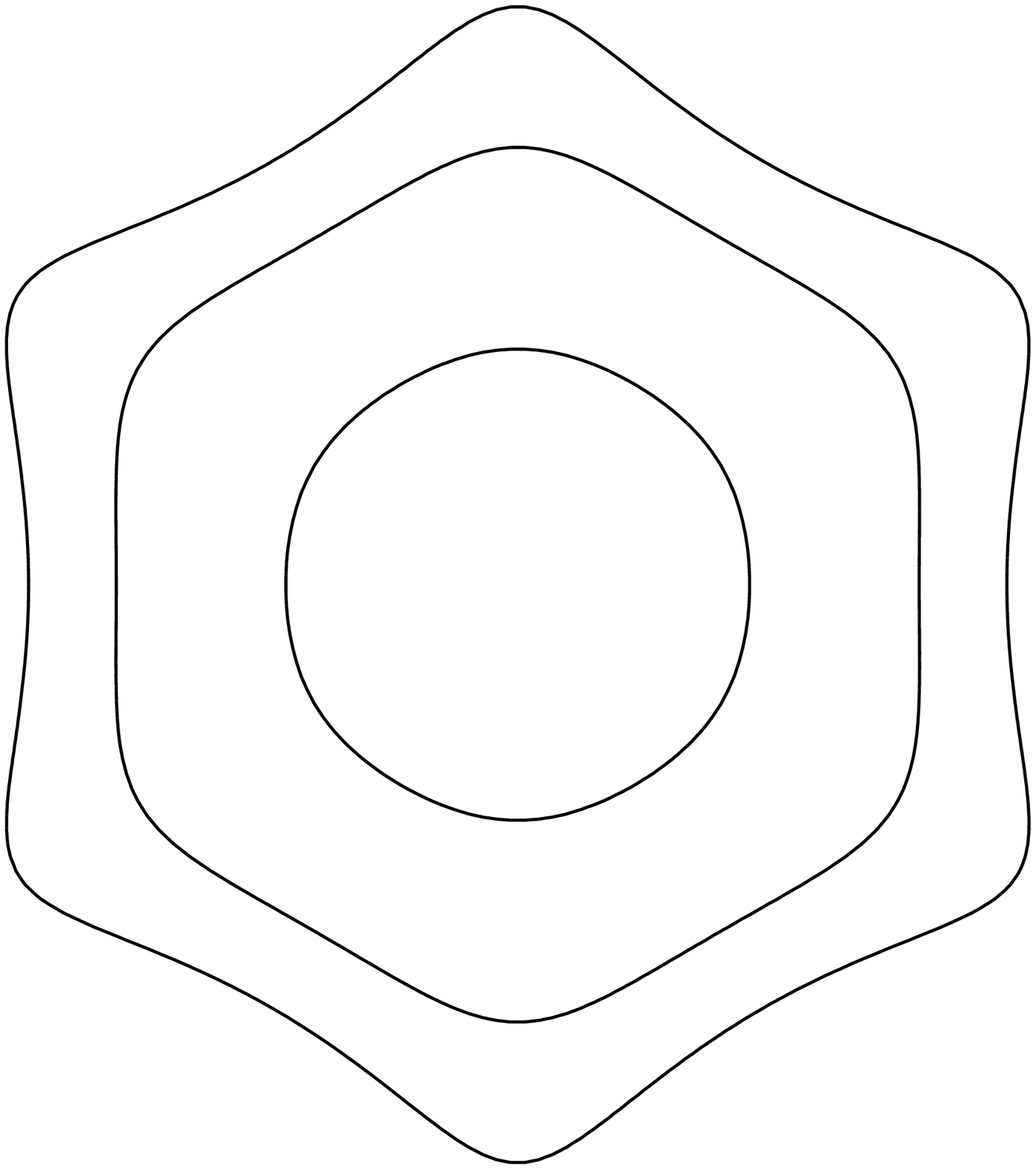}}
 \qquad\qquad
\subfigure[$\D_6$ with $u>0$]{\includegraphics[scale=0.15]{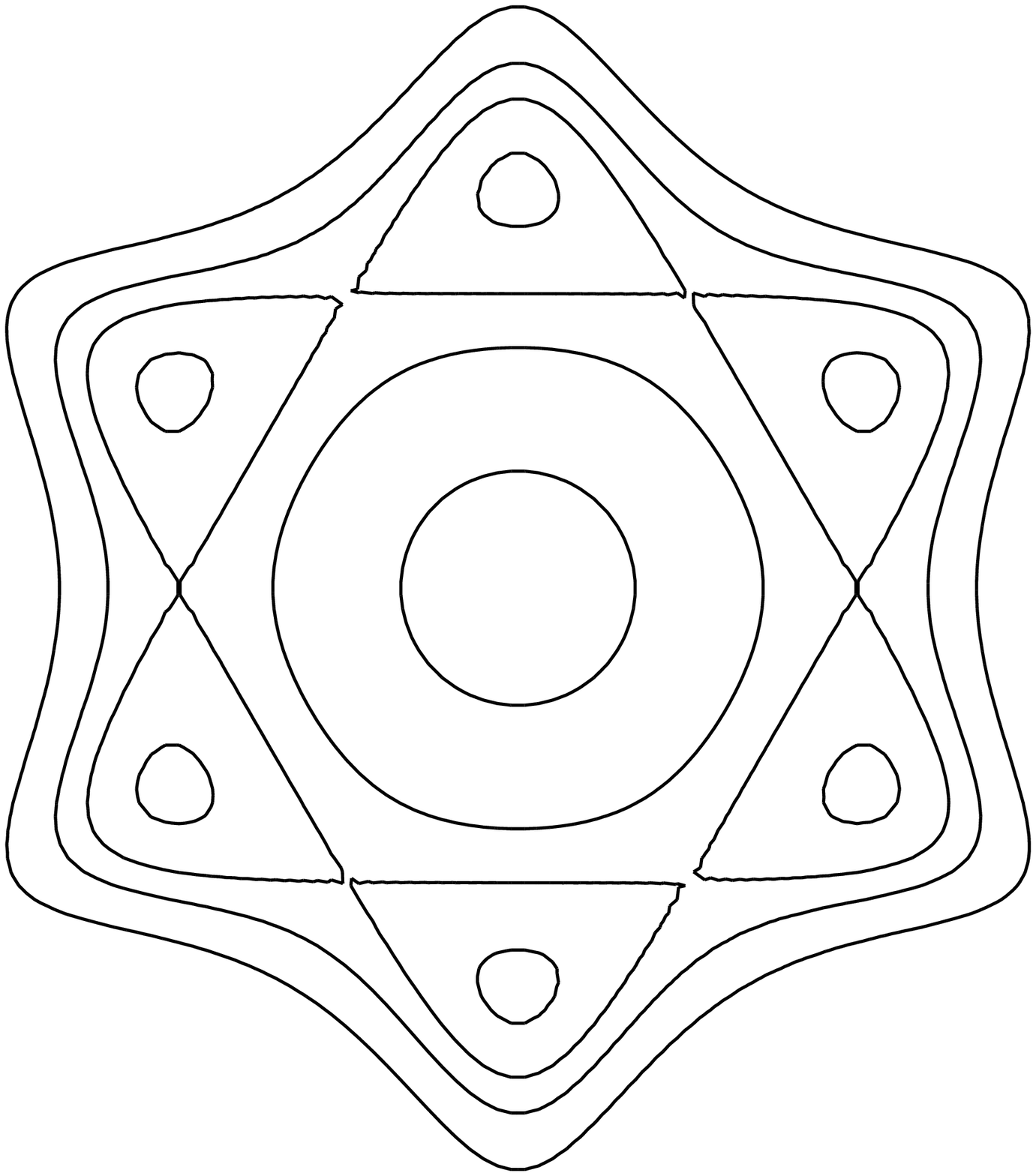}}

\caption{\textbf{(continued)} Contours of the generic 1-parameter family of $\D_k$-invariant functions (\ref{eq:generic Dn function}), for $k\geqslant 4$. These figures are all dihedral pitchfork bifurcations.}
\label{fig:dihedral bifurcations}
\end{center}
\end{figure}

\noindent$\bullet$ If $k=3$ or if $k=4$ with $\beta > \alpha$, then the bifurcation can be said to be transcritical, in that the bifurcating branches exist on both sides of the bifurcation point $u=0$, and the $k$ bifurcating points are all saddles (and hence unstable equilibria), each of which is fixed by a mirror reflection conjugate to $m$. See the bifurcation diagrams in Figure \ref{fig:dihedral bifurcation diagrams}(c, d).

\noindent$\bullet$ If $k=4$ and $\alpha > \beta$ or if $k>4$, then the bifurcation is like a pitchfork, in that all bifurcating equilibria coexist on the same side of the bifurcation point. But unlike the pitchfork, 2 types of bifurcating solutions appear, possibly with different stability properties; if $k$ is even, then one has symmetry type $\left< m \right>$ and the other  $\left< m' \right>$.  See the bifurcation
diagram in Figure \ref{fig:dihedral bifurcation diagrams}(e).

\begin{remark}
The finite determinacy and unfolding theorems of singularity theory guarantee that $f_0$ in (\ref{eq:generic D2 function}) and (\ref{eq:generic Dn function}) is finitely determined and h.o.t.\ may be ignored. If $n\leqslant 3$, then $f_0$ has codimension 1 provided $\beta \neq 0$, and $f_u$ is a versal unfolding of $f_0$, so that any deformation is equivalent to it.
If $n\geqslant 4$, then $f_0$ has codimension 2, and a versal unfolding is
$$
f_{u,v} = -uN + (\alpha + v)N^2 + \beta P
$$
provided $\beta \neq 0$ (and $\alpha \neq \pm \beta$ when $n=4$).  The parameter $v$ defines a 
{\it topologically\/} trivial deformation, i.e.\ $v$ can be eliminated via a continuous change of coordinates rather than a smooth one; nevertheless this homeomorphism will be a diffeomorphism away from the origin, so critical points are preserved.  The modulus $v$ that arises when $n=4$ is related to the cross-ratio of the 4 lines making up $f^{-1}(0)$.
\end{remark}

\subsection{Bifurcations of vortex rings}
\label{sec:bifurcations of rings}

We work with the parameter $\lambda r_0^2 > -1$.  Recall that $r_0$ is the radius of the vortex ring measured on $\CC$ after the stereographic projection, and that it is related  by (\ref{eq:radius}) to the radius $a$ measured on $M_\lambda$.  The 0 curvature case is $\lambda=0$, and in the spherical case $\lambda>0$ the values $\lambda r_0^2$ and $1/\lambda r_0^2$ are equivalent as they represent antipodal rings on the sphere. We therefore let 
$\lambda r_0^2$ vary in the range $(-1,\,1]$.

Now for the main theorem.   Stability means Lyapunov stability modulo rotations (same as orbital stability in our situation).  Instability means full spectral instability, i.e.\ at least one of the eigenvalues is real and positive. The ring of $n=3$ vortices is always nonlinearly stable.

\begin{theorem} \label{thm:bifurcations} Let $n \geqslant 4$. With $\lambda r_0^2\in(-1,1]$ as a parameter, the regular ring of $n$ identical vortices undergoes the following bifurcations, illustrated in Figure \ref{fig:bifurcation diagrams}:
\begin{description}
\item[all $n$] The ring is stable for $\lambda r_0^2<b_n$, where $b_n$ is the unique root in $(-1,1]$ 
of \footnote{cf.\ (\ref{eq:stability criterion}).} 
\begin{equation}\label{eq:b_n}
 \frac{1+ b_n^2}{(1+ b_n)^2} = \frac1{2(n-1)}\Big\lfloor\frac{n^2}{4}\Big\rfloor .
\end{equation}
\item[$n$ even] As $\lambda r_0^2$ crosses $b_n$, the ring loses stability via a supercritical pitchfork bifurcation, and  the (stable) bifurcating solution consists of a pair of $n/2$-gons with different values for the radius.
\item[$n$ odd] As $\lambda r_0^2$ crosses $b_n$, the ring loses stability via a supercritical bifurcation as depicted in Figure \ref{fig:dihedral bifurcation diagrams}(e), to 2 types of relative equilibria, each with a line of symmetry.
\item[all $n$]     As $\lambda r_0^2$ increases further, the ring undergoes a sequence of bifurcations, one in each of the modes $\lfloor n/2 \rfloor>\ell>1$; all the relative equilibria involved are unstable, and the bifurcating solutions have $\D_{(n,\ell)}$-symmetry.\footnote{We put $\D_1=\ZZ_2$ acting by reflection.}
\end{description}
\end{theorem}

\noindent As a special case, we recover the following result of Kurakin and Yudovich \cite{KY02} and Schmidt \cite{Schmidt04}.  The calculation justifying it is the subject of section \ref{sec:n odd}.

\begin{corollary}\label{coroll:Thomson}
The Thomson heptagon is nonlinearly stable.
\end{corollary}

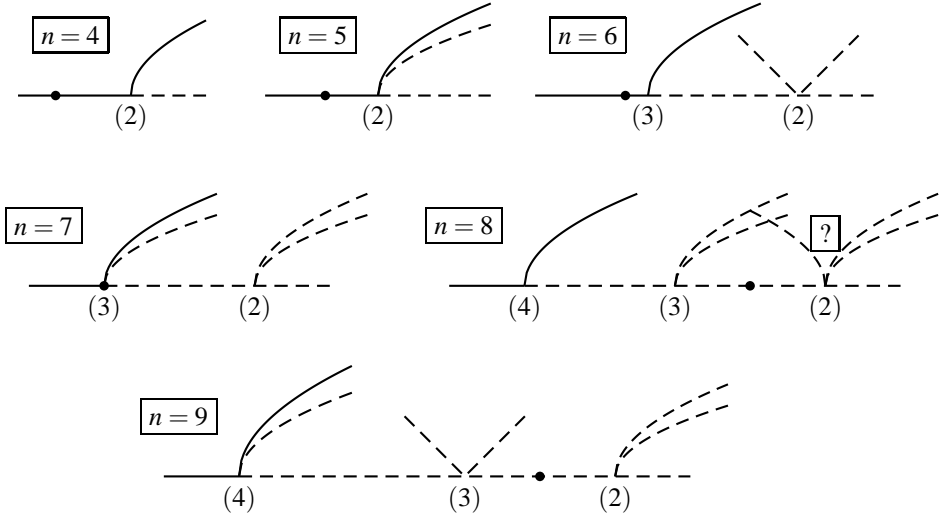
\begin{figure}[t]
\begin{center}
\begin{pspicture}(-1.5,-0.5)(1.7,1.5)
\rput(-0.8,0.8){\fbox{$n=4$}}
\psline(-1.5,0)(0,0) \psdot(-1,0) \psline[linestyle=dashed](0,0)(1,0)
\rput(0,-0.3){$(2)$}
\psplot{0}{1}{x 0.5 exp}
\end{pspicture}
\begin{pspicture}(-1.5,-0.5)(2,1.5)
\rput(-0.8,0.8){\fbox{$n=5$}}
\psline(-1.5,0)(0,0) \psdot(-0.7,0) \psline[linestyle=dashed](0,0)(1.5,0)
\rput(0,-0.3){$(2)$}
\psplot[linestyle=dashed]{0}{1.5}{x 0.5 exp 1.3 div}
\psplot{0}{1.5}{x 0.5 exp}
\end{pspicture}
\begin{pspicture}(-1.5,-0.5)(3,1.5)
\rput(-0.8,0.8){\fbox{$n=6$}}
\psline(-1.5,0)(0,0) \psdot(-0.3,0) \psline[linestyle=dashed](0,0)(3,0)
\rput(0,-0.3){$(3)$}
\psplot{0}{1.5}{x 0.5 exp}
\rput(2,-0.3){$(2)$}
\psline[linestyle=dashed](1.2,0.8)(2,0)(2.8,0.8)
\end{pspicture}

\begin{pspicture}(-1.5,-0.5)(4,2)
\rput(-0.8,0.8){\fbox{$n=7$}}
\psline(-1,0)(0,0) \psdot(0,0)\psline[linestyle=dashed](0,0)(3,0)
\rput(0,-0.3){$(3)$}
\psplot[linestyle=dashed]{0}{1.5}{x 0.5 exp 1.3 div}
\psplot{0}{1.5}{x 0.5 exp}
\rput(2,-0.3){$(2)$}
\psplot[linestyle=dashed]{2}{3.5}{x 2 sub 0.5 exp 1.3 div}
\psplot[linestyle=dashed]{2}{3.5}{x 2 sub 0.5 exp}
\end{pspicture}
\begin{pspicture}(-1.5,-0.5)(5,2)
\rput(-0.8,0.8){\fbox{$n=8$}}
\psline(-1,0)(0,0) \psdot(3,0) \psline[linestyle=dashed](0,0)(5,0)
\rput(0,-0.3){$(4)$}
\psplot{0}{1.5}{x 0.5 exp}
\rput(2,-0.3){$(3)$}
\psplot[linestyle=dashed]{2}{3.5}{x 2 sub 0.5 exp 1.3 div}
\psplot[linestyle=dashed]{2}{3.5}{x 2 sub 0.5 exp}
\rput(4,-0.3){$(2)$}
\psplot[linestyle=dashed]{4}{5.5}{x 4 sub 0.5 exp 1.3 div}
\psplot[linestyle=dashed]{4}{5.5}{x 4 sub 0.5 exp}
\psplot[linestyle=dashed]{3}{4}{x neg 4 add 0.5 exp}
 \rput(4,0.7){\fbox{?}}
\end{pspicture}

\begin{pspicture}(-1.5,-0.5)(5,2)
\rput(-1.8,0.8){\fbox{$n=9$}}
\psline(-2,0)(-1,0) \psdot(3,0) \psline[linestyle=dashed](-1,0)(5,0)
\rput(-1,-0.3){$(4)$}
\psplot{-1}{0.5}{x 1 add 0.5 exp 1.2 mul}
\psplot[linestyle=dashed]{-1}{0.5}{x 1 add 0.5 exp 1.1 div}
\rput(2,-0.3){$(3)$}
\psline[linestyle=dashed](1.2,0.8)(2,0)(2.8,0.8)
\rput(4,-0.3){$(2)$}
\psplot[linestyle=dashed]{4}{5.5}{x 4 sub 0.5 exp 1.3 div}
\psplot[linestyle=dashed]{4}{5.5}{x 4 sub 0.5 exp}
\end{pspicture}
\end{center}
\caption{Bifurcation diagrams for the ring of $n$ identical vortices for low values of $n$. The number in parentheses is the mode number bifurcating at that point. The black dot on the axis represents schematically the point where $\lambda=0$ (the plane).  $\lambda$ increases toward right. We do not know whether the bifurcating branches from the lower modes branch to the right or the left, though we believe they are as shown. The case $n=8$, $\ell=2$ has an effective action of $\D_4$, so could be transcritical or pitchfork---we do not know which occurs.}
\label{fig:bifurcation diagrams}
\end{figure}

\paragraph{Bifurcation values of $\lambda r_0^2$} 
The  tables below spell out the values of $\lambda r_0^2$ where the bifurcations occur, for $n=6,7,8, 9$. They are found by solving $\epsilon_r^{(\ell)}=0$ (\ref{eq:eigenvalues}) for $\lambda r_0^2$; the first values are those of $b_n$ mentioned in Theorem \ref{thm:bifurcations}.

$$\begin{array}{c|cc}
\multicolumn{3}{c}{n=6}\\[4pt]
\mbox{mode} & \;\ell=3 \;& \;\ell=2 \\
\hline
\rule{0pt}{10pt}\lambda r_0^2 & 0.056 &0.127
\end{array}
\qquad
\begin{array}{c|cc}
\multicolumn{3}{c}{n=7}\\[4pt]
\mbox{mode} & \;\ell=3 \;& \;\ell=2 \\
\hline
\rule{0pt}{10pt}\lambda r_0^2 & 0 &0.101
\end{array}
$$

\bigskip

$$
\begin{array}{c|ccc}
\multicolumn{4}{c}{n=8}\\[4pt]
\mbox{mode} &\ell=4 & \ell=3 & \ell=2 \\
\hline
\rule{0pt}{10pt}\lambda r_0^2 & -0.063 &-0.033 &\;0.084
\end{array}
\qquad\begin{array}{c|ccc}
\multicolumn{4}{c}{n=9}\\[4pt]
\mbox{mode} &\ell=4 & \ell=3 & \ell=2 \\
\hline
\rule{0pt}{10pt}\lambda r_0^2 & -0.101 & -0.056 &\; 0.072
\end{array}
$$

\bigskip

\noindent In all the tables, the bifurcation of the $\ell=2$ mode occurs for $\lambda>0$ (on the sphere);  this is easily checked to be true for all $n$.

\subsection{Geometry of bifurcating rings}
\label{sec:geometry of bifurcating rings}
At a bifurcation, the bifurcating mode controls the geometry/symmetry of the bifurcating solution.  Points in $V_\ell$ all correspond to configurations with cyclic symmetry $\ZZ_{(n,\ell)} \subset \D_n$, which allows the $\D_n$-action on $V_\ell$ to factor through a 
$\D_n/\ZZ_{(n,\ell)}\simeq\D_k$-action, where 
$$
k=n/(n,\ell).
$$ 
Moreover, the bifurcating solutions are all fixed by a reflection in $\D_k$ (conjugate to $m$ or to $m'$), which implies that they are symmetric under a $\ZZ_{(n,\ell)}$-action and under a reflection, together giving a symmetry of 
$\D_{(n,\ell)}$.
This means that if $(n,\ell)>1$, then the configuration consists of $k$ rings of $(n,\ell)$ vortices in each. Typical deformed configurations with the correct symmetry in each mode for $n=3, \ldots, 8$ are shown in Figure \ref{fig:perturbations}.

In particular, when $n$ is even and $\ell=n/2$, the solutions have symmetry isomorphic to $\D_{n/2}$. Now in $\D_n$ sit 2 non-conjugate copies of $\D_{n/2}$, one containing $m$, the other containing $m'$, and since $\zeta_r^{(n/2)}$ is fixed by $m$, the bifurcating solutions must have the symmetry $\D_{n/2}$ containing $m$.  Consequently the bifurcating solution consists of a pair of regular $n/2$-gons, in general of different radii, staggered by $2\pi/n$ as shown in Figures \ref{fig:perturbations}(a,d,k).

When $\ell\neq n/2$, the $\D_n$-action factors through a $\D_k$-action, and all bifurcating solutions have reflexive symmetry of $m$ or $m'$ (as seen from Figure \ref{fig:dihedral bifurcations}). Now if $k$ is odd, the resulting reflections are conjugate, so a configuration fixed by $m$ will also be fixed by some conjugate of $m'$.  This fact is illustrated in Figure \ref{fig:perturbations:6,2} where $n=6$, $\ell=2$,  $k=3$. Indeed, as $\D_n$ has $n$ reflections while $\D_k$ has $k$, in the representation $\D_n\to\D_k$ we must have $(n,\ell)$ reflections in $\D_n$ that get identified, thereby fixing the same configurations. On the other hand, if $k$ is even (as in $n=8$, $\ell=2$, Figures \ref{fig:perturbations}(g,h), 
the nonconjugate $m$ and $m'$ in $\D_n$ have as images 2 nonconjugate reflections in
$\D_k$, so the latter's fixed-point sets correspond to different configurations.

Finally, whenever $\ell$ divides $n$, a perturbation in the $\zeta_r^{(\ell)}$-direction produces $n/\ell$ rings of $\ell$-gons, in general of slightly different radii, and in the configurations with reflexive symmetry $m$  the vortices in the different $\ell$-gons line up with the original $n$-gon.

\begin{figure}[t]
\begin{center}
\psset{unit=1.5,dotsep=2pt,linewidth=0.5pt,dotsize=4pt}

\subfigure[$n=4,\:\ell=2$]{\begin{pspicture}(-1,-1)(1,1) %show(4,2,[0.2,0])
 \psline[linestyle=dotted,linewidth=1pt,linewidth=1pt](0., 1.)(-1., 0.)(0., -1.)(1., 0.)(0., 1.)
\psline[showpoints=true](0., .8)(-1.2, 0.)(0., -.8)(1.2, 0.)(0., .8)
\psline[linecolor=lightgray](-0.3,0)(0.3,0)
\psline[linecolor=lightgray](0,-0.3)(0,0.3)
\end{pspicture}\label{fig:perturbations:4,2}}\qquad\qquad
\subfigure[$n=5,\:\ell=2$]{\begin{pspicture}(-1,-1)(1,1) %show(5,2,[.2, 0, 0, .1])
 \psline[linestyle=dotted,linewidth=1pt](1., 0.)(.306, .952)(-.809, .588)(-.809, -.588)(.306, -.952)(1., 0.)
\psline[showpoints=true](.312, .780)(-.914, .546)(-.914, -.546)(.312, -.780)(1.2, 0.)(.312, .780)
\psline[linecolor=lightgray](-0.3,0)(0.3,0)
\end{pspicture}}

\subfigure[$n=6,\:\ell=2$]{\begin{pspicture}(-1,-1)(1,1) %show(6, 2, [.1, 0, 0, .05])
 \psline[linestyle=dotted,linewidth=1pt](.500, .865)(-.500, .865)(-1., 0.)(-.500, -.865)(.500, -.865)(1., 0.)(.500, .865)
\psline[showpoints=true](.512, .800)(-.512, .800)(-1.1, 0.)(-.512, -.800)(.512, -.800)(1.1, 0.)(.512, .800)
\psline[linecolor=lightgray](-0.3,0)(0.3,0)
\psline[linecolor=lightgray](0,-0.3)(0,0.3)
\end{pspicture}\label{fig:perturbations:6,2}}
\qquad\quad
\subfigure[$n=6,\:\ell=3$]{\begin{pspicture}(-1,-1)(1,1) %show(6,3,0.1,0)
 \psline[linestyle=dotted,linewidth=1pt](.500, .865)(-.500, .865)(-1., 0.)(-.500, -.865)(.500, -.865)(1., 0.)(.500, .865)
\psline[showpoints=true](.450, .778)(-.550, .952)(-.9, 0.)(-.550, -.952)(.450, -.778)(1.1, 0.)(.450, .778)
\psline[linecolor=lightgray](-0.3,0)(0.3,0)
\psline[linecolor=lightgray](-.1500, .2598)(.1500, -.2598)
\psline[linecolor=lightgray](-.1500, -.2598)(.1500, .2598)
\end{pspicture}\label{fig:perturbations:6,3}}

\subfigure[$n=7,\:\ell=2$]{\begin{pspicture}(-1,-1)(1,1) %show(7, 2, [.1, 0, 0, 0.05])
 \psline[linestyle=dotted,linewidth=1pt](.623, .782)(-.219, .976)(-.901, .434)(-.901, -.434)(-.219, -.976)(.623, -.782)(1., 0.)(.623, .782)
\psline[showpoints=true](.647, .735)(-.220, .883)(-.972, .425)(-.972, -.425)(-.220, -.883)(.647, -.735)(1.1, 0.)(.647, .735)
\psline[linecolor=lightgray](-0.3,0)(0.3,0)
\end{pspicture}}
\qquad\qquad
\subfigure[$n=7,\:\ell=3$]{\begin{pspicture}(-1,-1)(1,1) %show(7, 3, [.1, 0, 0, 0.05])
 \psline[linestyle=dotted,linewidth=1pt](.623, .782)(-.219, .976)(-.901, .434)(-.901, -.434)(-.219, -.976)(.623, -.782)(1., 0.)(.623, .782)
 \psline[showpoints=true](.584, .698)(-.270, 1.02)(-.860, .468)(-.860, -.468)(-.270, -1.02)(.584, -.698)(1.1, 0.)(.584, .698)
\psline[linecolor=lightgray](-0.3,0)(0.3,0)
\end{pspicture}}
\end{center}
\caption{Perturbations of the $n$-ring in mode $\ell$. These configurations are invariant under a subgroup isomorphic to $\D_{(n,\ell)}$, and the grey lines in the centre of each represent the lines of reflection. The dotted figures are the regular $n$-gons. See section \ref{sec:geometry of bifurcating rings} for explanations. Continued on next page.
}
\end{figure}
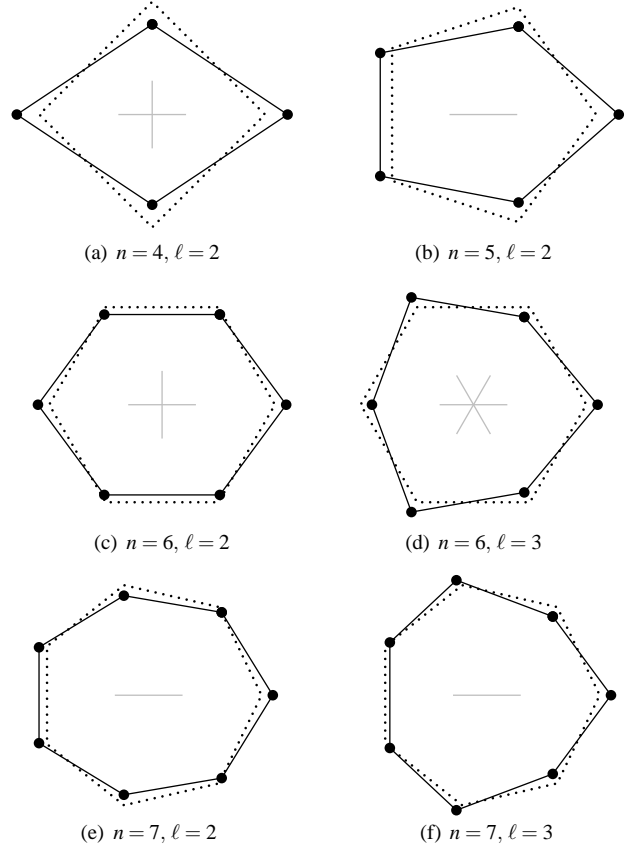

\begin{figure}
\addtocounter{figure}{-1}%\samenumber
\setcounter{subfigure}{6}
\begin{center}
\psset{unit=1.5,dotsep=2pt,linewidth=0.5pt,dotsize=4pt} %unit=1.8

\subfigure[$n=8,\:\ell=2$ (Fix\,$m$)]{\begin{pspicture}(-1,-1)(1,1) % show(8, 2, [.1, 0, 0, 0.05])
 \psline[linestyle=dotted,linewidth=1pt](.705, .705)(0., 1.)(-.705, .705)(-1., 0.)(-.705, -.705)(0., -1.)(.705, -.705)(1., 0.)(.705, .705)
 \psline[showpoints=true](.740, .670)(0., .9)(-.740, .670)(-1.1, 0.)(-.740, -.670)(0., -.9)(.740, -.670)(1.1, 0.)(.740, .670)
 \psline[linecolor=lightgray](-0.3,0)(0.3,0)
\psline[linecolor=lightgray](0,-0.3)(0,0.3)
\end{pspicture}\label{fig:perturbations:8,2a}}
\qquad\quad
\subfigure[$n=8,\:\ell=2$ (Fix\,$m'$)]{\begin{pspicture}(-1,-1)(1,1) %show(8, 2, [.1*cos((1/4)*Pi), -.1*sin((1/4)*Pi), 0.5e-1*sin((1/4)*Pi), -0.5e-1*cos((1/4)*Pi)])
 \psline[linestyle=dotted,linewidth=1pt](.705, .705)(0., 1.)(-.705, .705)(-1., 0.)(-.705, -.705)(0., -1.)(.705, -.705)(1., 0.)(.705, .705)
 \psline[showpoints=true](.729, .779)(-0.0352, .930)(-.631, .681)(-1.07, 0.0352)(-.729, -.779)(0.0352, -.930)(.631, -.681)(1.07, -0.0352)(.729, .779)
\psline[linecolor=lightgray](-.277, -.115)(.277, .115)
\psline[linecolor=lightgray](.115,-.277)(-.115,.277)
\end{pspicture}\label{fig:perturbations:8,2b}}

\subfigure[$n=8,\:\ell=3$ (Fix\,$m$)]{\begin{pspicture}(-1,-1.1)(1,1) % show(8, 3, [.1, 0, 0, 0.05])
 \psline[linestyle=dotted,linewidth=1pt](.705, .705)(0., 1.)(-.705, .705)(-1., 0.)(-.705, -.705)(0., -1.)(.705, -.705)(1., 0.)(.705, .705)
\psline[showpoints=true](.681, .631)(-0.05, 1.)(-.729, .779)(-.9, 0.)(-.729, -.779)(-0.05, -1.)(.681, -.631)(1.1, 0.)(.681, .631)
\psline[linecolor=lightgray](-0.3,0)(0.3,0)
\end{pspicture}}
\qquad\quad
\subfigure[$n=8,\:\ell=3$ (Fix\,$m'$)]{\begin{pspicture}(-1,-1.1)(1,1) % show(8, 3, [.1*cos(3*Pi*(1/8)), -.1*sin(3*Pi*(1/8)), 0.5e-1*sin(3*Pi*(1/8)), -0.5e-1*cos(3*Pi*(1/8))])
 \psline[linestyle=dotted,linewidth=1pt](.705, .705)(0., 1.)(-.705, .705)(-1., 0.)(-.705, -.705)(0., -1.)(.705, -.705)(1., 0.)(.705, .705)
\psline[showpoints=true](.701, .765)(0.019, .908)(-.762, .790)(-.962, -0.0462)(-.712, -.648)(0.0190, -1.09)(.654, -.626)(1.04, -0.0462)(.701, .765)
\psline[linecolor=lightgray](-.277, -.115)(.277, .115)
\end{pspicture}}
\qquad\quad
\subfigure[$n=8,\:\ell=4$]{\begin{pspicture}(-1,-1.1)(1,1) % show(8,4,[0.1, 0])
 \psline[linestyle=dotted,linewidth=1pt](.705, .705)(0., 1.)(-.705, .705)(-1., 0.)(-.705, -.705)(0., -1.)(.705, -.705)(1., 0.)(.705, .705)
\psline[showpoints=true](.634, .634)(0., 1.1)(-.634, .634)(-1.1, 0.)(-.634, -.634)(0., -1.1)(.634, -.634)(1.1, 0.)(.634, .634)
\psline[linecolor=lightgray](-0.3,0)(0.3,0)
\psline[linecolor=lightgray](0,-0.3)(0,0.3)
\psline[linecolor=lightgray](-0.212,0.212)(0.212,-0.212)
\psline[linecolor=lightgray](0.212,0.212)(-0.212,-0.212)
\end{pspicture}\label{fig:perturbations:8,4}}

\caption{(continued)
%\textbf{(continued)}
}
\label{fig:perturbations}
\end{center}
\end{figure}
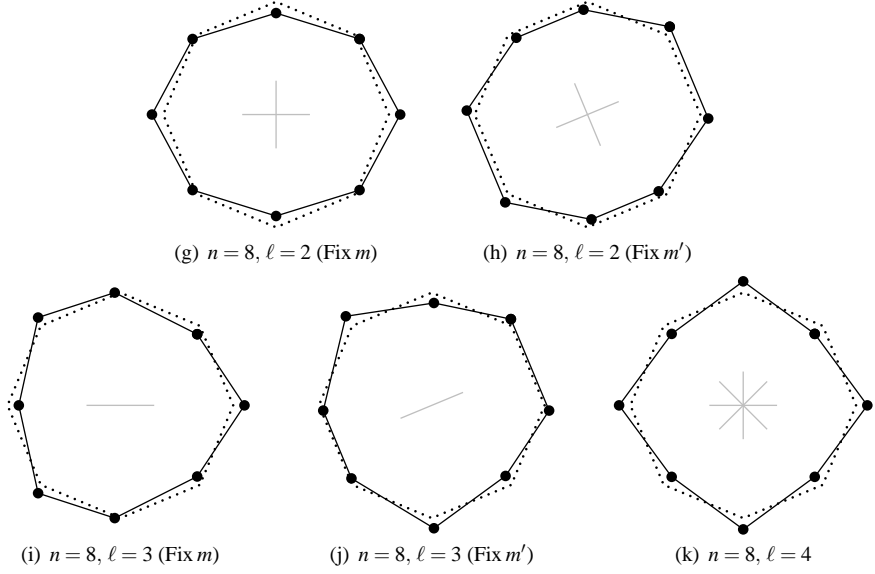

%%%%%%%%%%%%%%%%%%%%%%
\subsection{Degenerate critical points} \label{sec:degenerate}

This section is dedicated to proving Theorem \ref{thm:bifurcations}. We begin by presenting a criterion for a degenerate critical point to be a local minimum, then in sections \ref{sec:n even} and \ref{sec:n odd} respectively apply the criterion to the cases where $n$ is even and $n$ is odd.

\begin{lemma} \label{lemma:probes}
Let $f$ be an analytic function defined on a neighbourhood of $0$ in $\RR^n$ with a degenerate critical point at $0$, such that $f(0)=0$.  Write $B=\dd^2\!f(0)$, $C=\dd^3\!\!f(0)$, $D=\dd^4\!\!f(0)$.
If $B$ is positive-semidefinite and if for all $\aa\in {\rm Ker}\, B \setminus \{ 0 \}$, $\bb\in\RR^n$ we have
$$
C\aa^3=0,\quad D\aa^4+6C\aa^2\bb + 3B\,\bb^2 >0,
$$
then $f$ is strictly positive on a punctured neighbourhood of\/ $0$. 
\end{lemma}
The symbol like $C\aa^2\bb$ means `evaluate the trilinear form $C$ at $\aa$ in 2 of its 3 arguments and at $\bb$ in the 1 remaining argument'.
In our application of this lemma, $f$ will be the augmented Hamiltonian $\widehat{H}_{\lambda}$
(section \ref{sec:Hessians}), and for the
point-vortex problem this is analytic.

\begin{proof}
Suppose for a contradiction that $f^{-1}(0)$ intersects every punctured neighbourhood of $0$.
 
Use the splitting lemma (e.g.\ \cite{PS}) to write
$$f
(x,y) = Q(x) + h(y),
$$
where $Q$ is a homogeneous quadratic form and $h$ is a function with vanishing 2-jet. 
Explicitly, $Q(x) = \half x^TB\, x$ and $y\in {\rm Ker}\, B$. Since $y\in {\rm Ker}\, B$, the hypothesis $C\aa^3=0$ implies that $h$ in fact has vanishing 3-jet. Since the hypotheses of the lemma on the derivatives are \emph{intrinsic} (invariant under change of coordinates), in the new coordinates
$x, y$ they become $\dd^3 h=0$ and $\dd^4h\, \aa^4>0$ for all 
$\aa\in {\rm Ker}\, B \setminus \{ 0 \}$.

Use the curve selection lemma \cite{DLS,Hi73} to deduce the existence of an analytic curve $\gamma(t)$ passing through $0$ along which $f$ vanishes.
In the new coordinates, write $\gamma(t)= (x, y) = (\xi(t),\eta(t))$; then $Q(\xi(t))+h(\eta(t)) = 0$. Expand $\xi$, $\eta$ in Taylor series:
$$
\xi(t) = \xi_1t^r + \xi_2t^{r+1}+\cdots, \quad
\eta(t) = \eta_1t^r + \eta_2t^{r+1}+\cdots,
$$
where $r$ is the order of the curve (at least one of $\xi_1,\eta_1$ is nonzero). The leading terms of $f\circ (\xi, \eta)$, all of which must vanish, are
$$
Q(\xi(t)) = \half B\xi_1^2\, t^{2r} + B\xi_1\xi_2\, t^{2r+1}+\cdots
$$
By inspection this forces $\xi_1=\cdots=\xi_r=0$. The coefficient of $t^{4r}$ is then 
$B\xi_{r+1}^2 + \tfrac1{4!}\dd^4h\,\eta_1^4$, which must vanish.  But the order of the curve being $r$ and $\xi_1 = 0$, we must have $\eta_1\neq0$,
 hence from the hypothesis  
$B\xi_{r+1}^2 + \tfrac1{4!}\dd^4h\, \eta_1^2 > 0$,  a contradiction. \qed
\end{proof}

Consider now any Hamiltonian system on any symplectic manifold $M$, with Hamiltonian $H$, symmetry group $G$, momentum map $\J:M\to\gg^*$, and suppose $x\in M$ lies on a relative equilibrium with finite stabilizer (possibly trivial).  Define $T_0 := \gg\cdot x\,\cap\, {\rm Ker}\,\dd \J$.  The symplectic slice $\NN$ is then any $G_x$-invariant complement to $T_0$ in ${\rm Ker}\, \dd \J$.

\begin{proposition} \label{prop:degenerate stability}
If the hypotheses of Lemma \ref{lemma:probes} are satisfied with $\RR^n$ replaced by $\NN$ and $f$ by $\widehat{H}$, then the relative equilibrium is Lyapunov stable. 
\end{proposition}

\begin{proof}
It is enough to show that in the reduced space $M_\mu$, $\mu=\J(x)$, the reduced Hamiltonian admits a local extremum at $x$, cf.\ \cite{Mo97}. Since the action is locally free, 
$\J$ is a submersion near $x$.  The tangent space of the submanifold $\J^{-1}(\mu)$ is $T_0 \oplus \NN$, and $\dd\J_\lambda(x)$ maps $\NN$ isomorphically to $T_{\bar x}M_\mu$, where 
$\bar x$ is the image of $x$ in $M_\mu$.
Let $\bar\gamma$ be a curve in $M_\mu$ through $\bar x$. Then $\bar\gamma$ lifts to a smooth curve in $M$ tangent to $\NN$ at $x$. The claim now follows because the hypotheses of Lemma \ref{lemma:probes} applied on $M_\mu$ are equivalent to the same hypotheses applied on $\NN$. \qed
\end{proof}

As the argument and calculations for the stability of the bifurcating points take distinct turns depending on the parity of $n$, we treat the even and odd cases separately. The even case is fairly easy, the odd case is much harder.

\subsubsection{$n$ even} \label{sec:n even}
The critical mode is $\ell = n/2$, we have $\zeta^{(n/2)} = \alpha^{(n/2)}$, and both 
$c$ and $m$ act by multiplication by $-1$. This means that on this Fourier mode 
$V_{n/2} = \left< \alpha^{(n/2)}_r, \alpha^{(n/2)}_\theta \right>$, the augmented Hamiltonian 
$\widehat{H}_{\lambda}$ is an even function. We therefore expect, for generic families of functions,  $\widehat{H}_{\lambda}$ restricted to $V_{n/2}$ to be equivalent to a family of the form
$$
f_u(x,y) = \pm x^4 +u x^2 + y^2,
$$
where $u$ is a parameter depending on $r_0$, $\lambda$.  The $+$ sign in front of the $y^2$ term is justified because the $y$-direction here corresponds to $\alpha^{(\ell)}_\theta$, whose eigenvalue from (\ref{eq:eigenvalues}) is $\kappa^2 n^2 / 16\pi > 0$.

The $\D_n$-invariance of $\widehat{H}_{\lambda}$ helps us to figure out which terms arise in its Taylor series.  
For example, if $f\in V_\ell^*$, $g\in V_m^*$, and $fg$ is invariant, then $m = \ell$.

\begin{lemma}\label{lemma:Ca2b n even}
  For all $\aa\in V_{n/2}$ and $\bb\in \NN$, we have $C\aa^2\bb=0$.
\end{lemma}

\begin{proof}
On the symplectic slice expand $\widehat{H}_{\lambda}$ in Taylor series.  Each term is invariant, in particular the 3rd-order term $C\mathbf{x}^3$ for $\mathbf{x}\in\NN$.  Given $\aa_i\in V_{\ell_i}$ ($i=1,2,3$), the quantity $C\aa_1\aa_2\aa_3$ lies in the tensor product $V_{\ell_1}\otimes V_{\ell_2}\otimes V_{\ell_3}$, which contains invariant functions if and only if $\ell_1\pm\ell_2\pm\ell_3\equiv0\bmod n$.  For $\aa_1=\aa_2=\aa\in V_{n/2}$, $C\aa^2\bb$ (necessarily invariant) can be nonzero only if $\bb\in V_0$.  But $V_0\cap\NN=\{0\}$, implying that if $\bb\in\NN$ and $C\aa^2\bb$ is invariant, then $\bb=0$, so that $C\aa^2\bb=0$. \qed
\end{proof}

It remains to calculate $D\aa^4$ for $\aa\in V_{n/2}$ in order to apply Proposition \ref{prop:degenerate stability}. 
The criterion for a bifurcation (\ref{eq:b_n}) reads, for even $n$,
\begin{equation}\label{eq:bifurcation for n even}
 \frac{1+\lambda^2r_0^4}{(1+\lambda r_0^2)^2} = \frac{n^2}{8(n-1)}.
\end{equation}

\medskip

\noindent   Put $z_j = (r_0+(-1)^j\, t)\ee^{2\pi \ii j/n}$ and $f(t) = \widehat{H}_{\lambda}(z_1,\dots,z_n)$. 
We shall expand
$$
f(t) = -\frac1{4\pi}\sum_{i<j}\log|z_i-z_j|^2 + \frac{n-1}{4\pi}\sum_j \log(1+\lambda|z_j|^2)
- \omega\sum_j\frac{|z_j|^2}{1+\lambda|z_j|^2}.
$$
to the 4th order in $t$.

\paragraph{Calculation}
At the bifurcation point, $r_0$ and $\lambda$ are related by (\ref{eq:bifurcation for n even}) and $\omega=\omega_{\, 0}$ is given by (\ref{eq:angular velocity}).  We are expanding
\begin{eqnarray*}
f(t) &=& -\frac{n}{8\pi}\>  \sum_{1 \leqslant k \leqslant n - 1, \; k \mathrm{~odd}}
\log\left(r_0^2+t^2 - (r_0^2-t^2)\cos(2\pi k/n)\right)\\
&&\quad- \frac{n(n-2)}{32\pi} \left( \log\left(r_0+t\right)^2+\log\left(r_0-t\right)^2\right) \\
&&\quad +\frac{n(n-1)}{8\pi}\left( \log(1+\lambda(r_0+t)^2) + \log(1+\lambda(r_0-t)^2)\right) \\
&&\qquad -\omega_{\, 0}\frac{n}2\left( \frac{(r_0+t)^2}{1+\lambda(r_0+t)^2} + \frac{(r_0-t)^2}{1+\lambda(r_0-t)^2}\right) \quad +\quad \cdots
\end{eqnarray*}
where $\cdots$ is a constant independent of $r_0$, $t$, $n$, which will henceforth be ignored.
Taking Taylor series in $t$ of all of these terms to order 4 is simple, except for the first line, which 
comes out as
$$
-\frac{n^2}{8\pi}\log r_0 - \frac{n^2(n-2)}{32\pi r_0^2}t^2 + \frac{n^2(n-2)(n^2+2n-12)}{768\pi r_0^4}t^4 + O(t^6)
$$
(up to an additive constant), thanks to identities akin to (\ref{eq:identity}).
The coefficient of $t^2$ in the Taylor series is then
$$
\frac{n}{32\pi r_0^2\sigma^2}\left( -(n-2)^2\sigma^2+4(n-1)(1-\lambda r_0^2)^2\right)
$$
which can be shown to vanish subject to the bifurcation relation (\ref{eq:bifurcation for n even}).  %
The coefficient of $t^4$ is
\begin{eqnarray*}
\frac{n}{768\pi r_0^4\sigma^4}&& \bigl[ (n-2)(n^3+2n^2-12n+24)\sigma^4 \\
&&\quad + 24(n-1)\lambda r_0^2(19\lambda^3r_0^6 - 54\lambda^2r_0^4+43\lambda r_0^2-4)\bigr].
\end{eqnarray*}
Denote by $T$ the term in square brackets. We wish to show that $T>0$ at the bifurcation point. 
Solving (\ref{eq:bifurcation for n even}) for $\lambda r_0^2$ yields 2 roots, substituting which into $T$ in turn yields 2 values, say $T_1,T_2$ (functions of $n$).   Write $m = n - 2$.  A calculation (using {\tt Maple}!) 
reveals that $T_1+T_2$ is equal to 
\begin{eqnarray*}
\frac{128(m+1)}{(m^2-4m-4)^4}&&\bigl( 384 +2560m+4m^9+4832m^2+5024m^3+10616m^4\\
&&\qquad +15888m^5+10778m^6+3266m^7+177m^8\bigr),
\end{eqnarray*}
while $T_1 T_2$ is equal to 
\begin{eqnarray*}
\frac{4096(m+1)^2}{(m^2-4m-4)^4} && \bigl(16m^{10} +648m^9+7409m^8+1044m^7+39960m^6+85512m^5\\
&&\qquad +57332m^4+25824m^3+15136m^2+5376m+576\bigr).
\end{eqnarray*}
In view of $m\geqslant 0$ and $n\geqslant 2$, both are manifestly strictly positive, so that each of $T_1$ and $T_2$ is indeed strictly positive. \hfill$\Box$

\begin{remark}\label{rk:supercritical}
Since the relative equilibrium is stable at the point of bifurcation, the resulting pitchfork bifurcations are \emph{supercritical\/}: the bifurcating relative equilibria are stable, and coexist with the unstable central one. Thus these stable bifurcating relative equilibria exist in a neighbourhood of the bifurcation point, $\lambda r_0^2$ satisfying (\ref{eq:stability criterion})
with the inequality reversed.  See also Figure \ref{fig:stability ranges}.

To persuade ourselves that this is a genuine pitchfork, we need to check the 
nondegeneracy condition which is that the eigenvalues of the Hessian move through $0$ at nonzero speed with respect to the parameter $\lambda r_0^2$ (or just $\lambda$ or $r_0$ separately).  The expression (\ref{eq:eigenvalues}) for the eigenvalues permits an easy check.

As the mode that bifurcates is $\ell=n/2$, the bifurcation occurs in the fixed-point space for the subgroup $\D_{n/2}$ as explained in section \ref{sec:geometry of bifurcating rings}. The bifurcating solutions have $\D_{n/2}$-symmetry, i.e.\ consist of 2 regular $n/2$-gons at slightly different radii from the common centre, and these bifurcating solutions are \emph{stable}, at least close to the bifurcation point.
\end{remark}

\subsubsection{$n$ odd}  \label{sec:n odd}
There are two reasons why $n$ odd is much harder than $n$ even. First, ${\rm Ker}\, B$ (degeneracy space) is 2-dimensional and its basis elements are less simple (namely $\zeta_r^{(n-1)/2}$ rather than $\zeta_r^{(n/2)}$).  Second, the 3rd derivative contributions are 
nonzero and no analogue of Lemma \ref{lemma:Ca2b n even} holds.
We proceed as far as we can with general odd $n$, and then specialize to numerical calculations for a few low values of $n$.

The criterion for a bifurcation (\ref{eq:b_n}) reads, for odd $n$,
\begin{equation}\label{eq:bifurcation for n odd}
\frac{1+\lambda^2r_0^4}{\left(1+\lambda r_0^2\right)^2} = \frac{n+1}8.
\end{equation}
The critical mode is $\ell=\half(n-1)$, and
$$
V_c := {\rm Ker}\, B = \left<\alpha^{((n-1)/2)}_r,\;\beta^{((n-1)/2)}_r\right>\subset V_{(n-1)/2}.
$$
We wish to apply Proposition \ref{prop:degenerate stability}, based on Lemma \ref{lemma:probes} with $\aa\in V_c$ and $\bb\in\NN$. The calculations are simplified by the following observations. Recall the definition of $V_1'$ from (\ref{eq: V1'}).

\begin{proposition}\label{prop: n even invariants}
\begin{enumerate}
\item No cubic invariant exists on $V_c\,$, consequently $C\aa^3=0$.
\item Up to scalar multiple, there exists a unique quartic invariant on $V_c\,$, 
consequently $D\aa^4$ is a multiple of $|\aa|^4$.
\item Up to scalar multiple, there exists a unique cubic invariant of the form $C\aa^2\bb$ with $\aa\in V_c$ and $\bb\in \NN$, and invariance forces $\bb\in V_1'$.
\end{enumerate}
\end{proposition}

\begin{proof} Because $\half(n-1)$ is coprime to $n$, the action of $\D_n$ on $V_c$ 
is equivalent to the usual representation of $\D_n$ in the plane, though with an unusual choice of generator, cf.\ (\ref{eq:action on modes}).  $\D_n$-invariant functions on $V_c$ are functions of $N=x^2+y^2$ and $P=\mathrm{Re}(x+\ii y)^n$, cf.\ comment just before (\ref{eq:generic Dn function}).
Write $\aa=(x,y)$.

(i) As $n\geqslant 5$, this representation accommodates no cubic invariants, and as $C\aa^3$ must be invariant, it is $0$.

(ii) Likewise, the unique quartic invariant on $V_c$ is $N^2$, so $D\aa^4$ is a scalar multiple of $N^2=|\aa|^4$.

(iii) If $\aa\in V_c$ and $\bb\in V_m$, then 
$C\aa^2\bb\in V_{2\cdot (n-1)/2 +m} \oplus V_{2\cdot (n-1)/2 - m}\oplus V_{m}$. 
For this to be invariant, we need $m=0$ or $m=1$.  The former is ruled out by the assumption $\bb\in \NN$, so $\bb\in \NN\cap V_1 = V_1'$. \qed
\end{proof}

To understand better the invariant $C\aa^2\bb$, let $x,y$ be as before on $V_c$ and $u,v$ be coordinates on $V_1'$, chosen so that $m\cdot(x,y)=(x,-y)$ and
$m\cdot(u,v)=(u,-v)$; in a nutshell $m\cdot (z,w)=(\bar z,\bar w)$ in terms of complex variables $z=x+\ii y$ and $w=u+\ii v$.  We then have $c\cdot (z,w) = (c^{(n-1)/2}z,\, c w)$. The cubics of the form  $C\aa^2\bb$ are the real and imaginary parts of $z^2w$, $z^2\bar w$, $|z|^2w$.  However, only the first of these is invariant under $c$, and only its real part is invariant under $m$.  Thus,
$$
C\aa^2\bb = \gamma\, (z^2w+\bar z^2\bar w) = 2\gamma \left(u(x^2-y^2)-2vxy\right)
$$
for some value of $\gamma\in\RR$; explicitly $\gamma=\frac14 \frac{\partial^3}{\rule{0pt}{8pt}\partial x^2\partial u} (C\aa^2\bb) = \frac32 \frac{\partial^3}{\rule{0pt}{8pt}\partial x^2\partial u}
\widehat{H}_{\lambda}$.

The key quantity $D\aa^4+C\aa^2\bb+B\bb^2$ becomes, 
on completing the square, 
\begin{equation}\label{eq:completing the square}
\delta |z|^4+2 \gamma\, \mathrm{Re}(z^2w)+ \beta |w|^2 =
\delta \left|z^2+\frac{\gamma}{\delta}\bar w\right|^2 + 
\frac{\beta \delta -\gamma^{\, 2}} \delta |w|^2.
\end{equation}
Manifestly this is positive for all $z$, $w\neq0$ if and only if $\delta >0$ and 
$\beta \delta >\gamma^{\, 2}$.  The value of $\beta$ is
$$
\beta =\epsilon_1'= \frac{n(n-1)\kappa^2}{4\pi r_0^2}\widetilde{\sigma}^{\, 2}
$$
where as before $\widetilde{\sigma}=1-\lambda r_0^2$, cf.\ (\ref{eq:rho1'}). There remains the task of 
calculating $\gamma$ and $\delta$. 

\paragraph{Calculation}
The awkward trigonometric expressions prevented us (and {\tt Maple}) from reaching closed forms for $\gamma$ and $\delta$.  We therefore proceed to evaluate them numerically.  In all these evalutations, $r_0$ is related to the parameter $\lambda$ by (\ref{eq:bifurcation for n odd}).  For $n=7$, of course $\lambda=0$ and $r_0$ is arbitrary.

\paragraph{$n=5\,$\rm:}
$$
\beta = \frac{37.1}{r_0^2},\quad \gamma = -\frac{9.04}{r_0^3}, \quad \delta = \frac{15.8}{r_0^4}.
$$
$\beta\delta - \gamma^{\, 2} = 504.7 / r_0^6 >0$ hence the pentagon is stable.

\paragraph{$n=7\,$\rm:}
It transpires ({\tt Maple}) that for $n=7$
we have
$$
\beta =\frac{21}{2\pi r_0^2},\quad \gamma = \frac{63}{4\pi r_0^3}, \quad 
\delta = \frac{1071}{8\pi r_0^4}.
$$
Computationally these numbers are correct to a high degree of precision---but we have no proof that they are rational multiples of $1/\pi r_0^k$.  At any rate it is certain that 
$\beta\delta - \gamma^{\, 2} >0$, hence the heptagon is stable. This establishes Corollary \ref{coroll:Thomson}.

\paragraph{$n=9\,$\rm:} 
$$
\beta =\frac{6.9}{r_0^2},\quad \gamma = \frac{13.8}{r_0^3}, \quad \delta = \frac{182.4}{r_0^4}.
$$
$\beta\delta - \gamma^{\, 2} = 1075.60 / r_0^4 >0$ hence the enneagon\footnote{`Nonagon' mixes Latin and Greek.} is stable.

\paragraph{$n=11\,$\rm :} 
$$
\beta =\frac{12.0}{r_0^2},\quad \gamma = \frac{29.7}{r_0^3}, \quad \delta = \frac{555.1}{r_0^4}.
$$
$\beta\delta - \gamma^{\, 2} = 5787.6 / r_0^6 >0$ hence the hendecagon\footnote{`Undecagon' is another Greco-Latin hybrid.} is stable.

\begin{conjecture}\label{conj:n-gon stable}
For all values of $n$, the $n$-gon is stable at the bifurcation point.
\end{conjecture}
Above, we have proved this for all even $n$ and for $n = 3, 5, 7, 9, 11$.

%%%%%%%%%%%%%%%%%%%%%%%%%%%%%%%%%%%%%
\subsection{Bifurcations from the equator}\label{sec:equator}

For all values of $n>3$, the $\ell=1$ mode `bifurcates' at $\lambda r_0^2=1$, i.e.\ when the ring of vortices lies on the equator of the sphere. The momentum value at such a ring is fixed by all of $\SO(3)$: indeed, for the usual coadjoint equivariant momentum map for this problem, the momentum value is $0$.  Let us summarize from \cite[Proposition 3.8]{LMR01} what bifurcations occur in this situation.  On each near-zero momentum sphere we get, besides the regular ring
of $n$ vortices, the following configurations. 
\begin{description}

\item[$n$ odd]  For each of the $n$ planes through the poles of the sphere and containing one of the vortices, 2 configurations consisting of $\half(n-1)$ pairs and that 1 vortex on the plane; 
the vortices in each pair are each other's reflection in that plane.
The notation in \cite{LMR01} is $C_h(\half(n-1) R,\,E)$,  $E$ referring to the single vortex on the plane, the $R$ to the reflection pairs.

\medskip

\item[$n$ even] In this case there are two distinct types of bifurcating solution, arising from the two distinct types of reflection in $D_n$:

\begin{enumerate} 

\item For each of the $\half n$ planes through the poles and containing a pair of diametrically 
opposite vortices, 1 configuration consisting of $\half n - 1$ reflection pairs and those 2 vortices on the plane.  The notation in \cite{LMR01} is $C_h((\half n -1)R,\,2E)$.

\item For each of the $\half n$ planes through the poles and passing midway between adjacent vortices, 1 configuration consisting of $\half n$ reflection pairs.  The notation in \cite{LMR01} is $C_h(\half n\, R)$.
\end{enumerate}
\end{description}

%%%%%%%%%%%%%%%%%%%%%%%%%%%%%%%%%%%%%
\section{What happens with other Hamiltonians}\label{sec:other Hamiltonians}

Toward the end of sections \ref{sec:surfaces} and  \ref{sec:vorticity}, we met three options for families of Green's functions, which all agree for the plane $\lambda=0$.  We have been opting for (\ref{eq:green's function from sphere}).  Here we sktech the conditions that guarantee the stability of the ring of identical
vortices for the other two options
(\ref{eq:green's function pole at infinity}), (\ref{eq:green's function opposite vortices}); the methods are 
the same as those of sections \ref{sec:vortices} and \ref{sec:bifurcations}.

\subsection{Green's function $G=\log|z-w|^2$}

\paragraph{Angular velocity} The regular ring of $n$ vortices with identical vorticities 
$\kappa$ rotates at angular velocity
$$\omega = -\frac{(n-1)\kappa}{8\pi}\frac{\sigma^2}{r_0^2},$$
Unlike the expression (\ref{eq:angular velocity}) this is not even in $\lambda$.

\paragraph{Stability}
The Hessian of the augmented Hamiltonian is the same as (\ref{eq:hessian}), except that $A$ is 
changed to
$$
A = - \frac{(n-1)\kappa^2}{24\pi r_0^2\sigma}\left( (n-11)+(n+13)\lambda r_0^2\right)
$$
and the eigenvalues of the Hessian become
$$
\epsilon_r^{(\ell)} = \frac{\kappa^2}{4\pi r_0^2}\left( (n-1)\frac{6+5\lambda^2 r_0^4}{3\sigma^2} - \ell(n-\ell)\right)
$$
while the expressions for $\epsilon_\theta^{(\ell)}$ are as before.  Also as before, the $\ell=\lfloor n/2 \rfloor$ mode has the least eigenvalue, so the ring is stable provided $\epsilon_r^{\lfloor n/2 \rfloor}>0$. The criterion is
$$
\frac{1+\frac{5}{6}\lambda^2r_0^4}{(1+\lambda r_0^2)^2} > \frac1{2(n-1)}\Big\lfloor \frac{n^2}4\Big\rfloor.
$$
Compared with (\ref{eq:stability criterion}), for each $n$ this new inequality is satisfied by a slightly narrower range of the effective parameter $\lambda r_0^2$. The transition from stable to unstable still occurs at $\lambda r_0^2$ of the same sign as for the previous Hamiltonian, and for $\lambda=0$ the two Hamiltonians agree.   Hence the Thomson heptagon is still stable.

\paragraph{Bifurcations} It seems likely that the bifurcations are of the same types as those explained in section \ref{sec:bifurcations}; we have checked this for $n=5,7,9$.

%%%%%%%%%%%%%%%%%%%%%%%%%%%%%%%%%
\subsection{Green's function $G=\log \frac{|z-w|^2}{|1+\lambda z\bar w|^2}$}

On the hyperbolic plane $\lambda < 0$ this reduces to the Hamiltonian adopted by Kimura \cite{Kimura}.  
For $\lambda = 0$ it is the standard Green's function on the plane, while for $\lambda > 0$ it corresponds to Green's function for the Laplacian on the sphere with `counter-vortices'.  The Hamiltonian will model $2n$ vortices placed pairwise at antipodal points, each pair having opposite vorticities. Thus the `ring' becomes 2 rings, one of $n$ vortices of vorticity $\kappa$ near the North Pole, the other of $n$ vortices of vorticity $-\kappa$ near the South Pole. In Laurent-Polz \cite{LP02} these configurations are referred to as $\D_{nh}(2R)$ when $n$ is even and $\D_{nd}(R,R')$ when $n$ is odd (in the former the rings are aligned, while in the latter they are staggered). The stability results of \cite{LP02} are not directly applicable here, as he considers stability with respect to  perturbations of all $2n$ vortices, whereas we are considering a restricted class of 
perturbations: those preserving the antipodal pairing of the configurations. If a configuration is stable for Laurent-Polz, then {\it a fortiori\/} it will be stable for our setting.

The calculations based on this option of Green's function get so cumbersome that the stability problem seems no longer tractable analytically.  It seems likely that the results are similar to those in section \ref{sec:bifurcations}, though the details of where the bifurcations occur will differ.  We did calculate that the angular velocity $\omega$ analogous to (\ref{eq:angular velocity}) of the ring is
$$
 -\frac{\kappa}{8\pi r_0^2}
\frac{1+\lambda r^2}{1-\left(-\lambda r^2\right)^n}\left((n-1)(1+\lambda r^2)\left(1+\left(-\lambda r^2\right)^{n-1}\right)+2\lambda r^2\left(1-\left(-\lambda r^2\right)^{n-1}\right)\right).
$$
For $\lambda>0$ (spheres) this can be deduced from \cite[Proposition 3.11]{LP02}.  

\paragraph{Acknowledgements}
JM thanks Tudor Ratiu and the staff of the Bernoulli Centre in Lausanne for their hospitality, as much of this paper was written during an extended visit there.  TT thanks L.~Mahadevan and the staff of SEAS at Harvard for their hospitality, as much of this paper was finished during an extended visit there.

%%%%%%%%%%%%%%%%%%%%%%%%%%%%%%%%%%%%%

\end{document}